\documentclass[10pt,leqno]{article}

\usepackage{amsmath,amssymb,amsthm,mathrsfs,dsfont,graphicx}

\usepackage[margin=3cm]{geometry} 

\usepackage{titlesec,hyperref}
\usepackage{esint}
\usepackage{natbib}
\usepackage{color}

\usepackage{fancyhdr}
\titleformat{\subsection}{\it}{\thesubsection.\enspace}{1pt}{}

\newtheorem{theo}{Theorem}[section]
\newtheorem{lemm}[theo]{Lemma}
\newtheorem{defi}[theo]{Definition}

\newtheorem{prop}[theo]{Proposition}
\newtheorem{rema}[theo]{Remark}
\numberwithin{equation}{section}

\allowdisplaybreaks 

\begin{document}
\title{Sharp non-uniqueness for the Boussinesq equation with fractional dissipation
\hspace{-4mm}
}

\author{Zipeng $\mbox{Chen}^1$ \footnote{Email: chenzp26@mail2.sysu.edu.cn}\quad
	 and\quad
	Zhaoyang $\mbox{Yin}^{1}$\footnote{E-mail: mcsyzy@mail.sysu.edu.cn}\\
    $^1\mbox{School}$ of Science,\\ Shenzhen Campus of Sun Yat-sen University, Shenzhen 518107, China}
        
\date{}
\maketitle
\hrule

\begin{abstract}
This paper focuses on the $d$-dimensional ($d\geq2$) Boussinesq equation with fractional dissipation $(-\Delta)^{\alpha}$ on the torus. We show that the uniqueness property breaks down within the function space  $L^p_tL^\infty_x$ for any $p<\frac{2\alpha}{2\alpha-1}$ when $1\leq\alpha<\frac{d+1}{2}$ and the function space $L^\frac{2\alpha}{2\alpha-1}_tL^q_x$ for any $q<\infty$ when $1<\alpha<\frac{d+1}{2}$. Moreover, the weak solutions we construct are smooth outside a set of singular times with Hausdorff dimension arbitrarily small. This result is sharp, as weak-strong uniqueness holds in the space $L^{\frac{2\alpha}{2\alpha-1}}_TL^\infty_x$. 
\end{abstract}

\vskip 0.2cm


\vspace*{10pt}


\section{Introduction }

  In this paper, we consider the following $d$-dimensional ($d\geq2$) Boussinesq equation:
\begin{equation}\label{e:boussinesq equation}
\begin{cases}
\partial_tv+ \text{div}(v\otimes v)+\nabla p+(-\Delta)^{\alpha} v=\theta e_d, \quad\quad \\
\text{div}\,v=0,\\
\partial_t\theta+\text{div}(v\theta)+(-\Delta)^{\alpha} \theta=0, \quad\forall(t,x)\in [0,T]\times\mathbb{T}^d,
\end{cases}
\end{equation}
where $T>0$, $\mathbb{T}^d$ is the $d$-dimensional torus and $e_d=(0,0,\cdots,1)$. Here, $v,p,\theta$ represent velocity, pressure, and temperature, respectively. The Boussinesq equation serves to model large-scale atmospheric and oceanic circulations underlying cold front formations and jet stream behaviors (see \cite{Majda}). Moreover, the Boussinesq equation is indispensable for investigating the Rayleigh–Bénard convection, which is one of the most widely studied convection phenomena (see \cite{convection2,convection1,convection3}). Let us first present the notion of weak solutions in the distributional sense to (\ref{e:boussinesq equation}).

\begin{defi}\label{def of ws}
    Let $(u_0,\theta_0)\in L^2(\mathbb{T}^d)$. A pair $(u,\theta)\in L_{t,x}^2$ is a weak solution of (\ref{e:boussinesq equation}) with initial data $(u_0,\theta_0)$ if the following hold: \begin{itemize}
        \item For a.e. $t\in[0,T]$, $u(t)$ is weakly divergence-free.
        \item For any $\phi\in C^\infty(\mathbb{R}\times\mathbb{T}^d;\mathbb{R}^d)$ with $\text{div}\,\phi=0$ and $\phi(t)=0$ if $t\geq T$,
        \begin{gather*}
            \int_{\mathbb{T}^d}u_0(x)\cdot\phi(0,x)\,dx=-\int^T_0\int_{\mathbb{T}^d}u\cdot(\partial_t\phi-(-\Delta)^{\alpha}\phi+u\cdot\nabla\phi)+\theta e_d\cdot\phi\, dxdt.
        \end{gather*}
        \item For any $\psi\in C^\infty(\mathbb{R}\times\mathbb{T}^d;\mathbb{R})$ with $\psi(t)=0$ if $t\geq T$, 
        \begin{gather*}
            \int_{\mathbb{T}^d}\theta_0(x)\cdot\psi(0,x)\,dx=-\int^T_0\int_{\mathbb{T}^d}\theta\cdot(\partial_t\psi-(-\Delta)^{\alpha}\psi+u\cdot\nabla\psi)\,dxdt.
        \end{gather*}
    \end{itemize}
    Moreover, we call the weak solution $(u,\theta)$ is generalized Leray-Hopf weak solution if $(u,\theta)$ belongs to $L^\infty_TL^2_x\cap L^2_T\dot{H}^\alpha_x$ and satisfies
    \begin{gather}
        \|u(t)\|^2_{L_x^2}+2\int_0^t\|(-\Delta)^{\frac{\alpha}{2}} u(s)\|^2_{L^2_x}\,ds\leq \|u_0\|^2_{L^2_x}+2\int^t_0\int_{\mathbb{T}^d}\theta e_d\cdot u\,dxds,\label{能量不等式1}\\
        \|\theta(t)\|^2_{L_x^2}+2\int_0^t\|(-\Delta)^{\frac{\alpha}{2}} \theta(s)\|^2_{L^2_x}\,ds\leq \|\theta_0\|^2_{L^2_x}.\label{能量不等式2}
    \end{gather}
\end{defi}

COnsider the famous Navier-Stokes equation:
\begin{equation*}
\begin{cases}
\partial_tv+ \text{div}(v\otimes v)+\nabla p-\nu\Delta v=0, \\
\text{div}\,v=0,\\
v(0)=v_0,
\end{cases}
\end{equation*}
Leray \cite{leray} first proved that there exists a weak solution in $L^\infty(\mathbb{R}^+;L^2(\mathbb{R}^d))\cap L^2(\mathbb{R}^+;\dot{H}^1(\mathbb{R}^d))$ for any $L^2$ solenoidal initial data, which satisfies the energy inequality
\begin{gather*}
    \|v(t)\|^2_{L^2}+2\nu\int^t_0\|\nabla v(s)\|^2_{L^2}ds\leq\|v_0\|^2_{L^2},\quad\forall t\geq0.
\end{gather*}
In the case of smooth bounded domain with Dirichlet boundary conditions, Hopf \cite{hopf} derived an analogous result. This category of weak solutions is nowadays known as Leray-Hopf weak solutions. In the present time, the uniqueness problem for Leray–Hopf weak solutions of the Navier–Stokes equations remains an open question. However, significant breakthroughs have been made in the uniqueness, regularity, and global existence of solutions to NS equation within broader function spaces.
The most relevant result to this article is the Ladyzhenskaya-Prodi-Serrin criteria \cite{kozono,lady,Prodi,serrin}. It states that if a Leray-Hopf weak solution belongs to the critical or subcritical regime:
\begin{gather*}
    L_t^pL_x^q\quad \text{with}\,\, \frac{2}{p}+\frac{d}{q}\leq 1,
\end{gather*}
then uniqueness is guaranteed for Leray–Hopf solutions with the same initial data (more precisely, we need $C_tL^q$ rather than $L_t^\infty L_x^q$ when $p=\infty$). Moreover, belonging to such functional spaces also implies the regularity of Leray–Hopf solutions \cite{w-s2,w-s1,w-s3}. For the 3D hyperdissipative NS equation, \cite{GLPS} proves that uniqueness holds in the critical or subcritical regime
\begin{gather*}
    L_t^pL_x^q\quad \text{with}\,\, \frac{2\alpha}{p}+\frac{3}{q}\leq 2\alpha-1.
\end{gather*}

 It's known that weak solution of NS equation will stay unique in the critical or subcritical regime due to LPS criteria. But in the supercritical, uniqueness may break down and there have been many non-uniqueness results. Buckmaster-Vicol \cite{ns有限能量不唯一} first proved the non-uniqueness of finite energy weak solutions to 3D NS equation, which is based on the convex integration scheme and the $L^2_x$-based intermittent spatial building block. Using the same method, Luo-Titi \cite{ns有限能量不唯一低于lion指标} extended the above result to 3D hyperviscous NS equation with fractional hyperviscosity exponent less than the Lion exponent $\frac{5}{4}$. Luo-Pu \cite{luopu} showed the non-uniqueness of $C_t^0L^2_x$ weak solutions to 2D hypoviscous NS equation.
 In \cite{serrin准则luo}, Cheskidov-Luo used the temporal intermittency method and proved the sharp non-uniqueness of NS equation on the class $L^p_tL^\infty_x$ when $p<2$. This is the first sharp non-uniqueness counterexample near the endpoint of LPS criteria. In \cite{MR4610908}, Cheskidov-Luo established sharp non-uniqueness on the class $L^{\infty}_tL^q_x$ when $q<2$ for the 2D NS equation, which is another endpoint of LPS citeria. For the 3D hyperdissipative NS equation with fractional hyperviscosity exponent beyond Lion exponent, Li-Qu-Zeng-Zhang \cite{qupeng} showed the sharp non-uniqueness in the supercritical space in view of the generalized LPS criteria.

In these papers of proving non-uniqueness results as well this one, the convex integration method stands as the most crucial tool. De Lellis-Sz\'{e}kelyhidi \cite{introductionconvex1,introductionconvex2} first introduced the convex integration scheme into the incompressible Euler equation. The convex integration scheme has been fully developed in recent years, for instance \citep{1,2,3,4,5,6,7,8,9}. The resolution of Onsager's conjecture marks a major milestone—specifically, Isett \cite{Isett3} resolved the conjecture for the 3D case (see also \cite{B1}), and Giri-Radu \cite{2donsager} achieved the same for the 2D case.

Driven by research on Onsager’s conjecture regarding the Euler equation and the non-uniqueness of weak solutions to the Navier-Stokes equation, a natural question arises: do analogous results hold for the Boussinesq system? For the 3D inviscid Boussinesq equation without thermal diffusion, Tao-Zhang \cite{T3} showed the existence of $C^{\beta}$ periodic weak solutions with the prescribed kinetic energy where $\beta\in(0,\frac{1}{5})$.  Miao-Nie-Ye \cite{Ye} proved the H\"{o}lder threshold regularity exponent for $L^p$-norm conservation of temperature of this system is $\frac{1}{3}$. The existence of periodic weak solutions belonging to $C^{\beta}$ with $\beta\in(0,\frac{1}{3})$ and possessing the prescribed kinetic energy was demonstrated by Xu-Tan \cite{xusaigu}. For the 2D Boussinesq equation with thermal diffusion, Luo-Tao-zhang \cite{L2,L1} constructed the H\"{o}lder continuous dissipative weak solutions with the prescribed kinetic energy and showed the non-uniqueness of finite energy weak solutions.

\section{Main result and outline of the paper}
The main theorems of this paper is formulated in Theorem \ref{maintheo} and Theorem \ref{maintheo2}, which shows non-uniqueness of weak solutions for the Boussinesq equation (\ref{e:boussinesq equation}).

\begin{theo}\label{maintheo}
    Let $d\geq2$ be the dimension and $1\leq\alpha<\frac{d+1}{2}$. For any $1\leq p<\frac{2\alpha}{2\alpha-1}$, if $(u,\theta)\in L^p_TL^\infty_x(\mathbb{T}^d)$ is a weak solution of (\ref{e:boussinesq equation}) and has at least one interval of regularity, then there exist infinitely many non-generalized-Leray-Hopf weak solutions in $L^p_TL^\infty_x(\mathbb{T}^d)$ of (\ref{e:boussinesq equation}) with the same initial data.
\end{theo}

\begin{theo}\label{maintheo2}
    Let $d\geq2$ be the dimension and $1<\alpha<\frac{d+1}{2}$. For any $q<\infty$, if $(u,\theta)\in L^\frac{2\alpha}{2\alpha-1}_tL^q_x$ is a weak solution of (\ref{e:boussinesq equation}) and has at least one interval of regularity, then there exist infinitely many non-generalized-Leray-Hopf weak solutions in $L^\frac{2\alpha}{2\alpha-1}_tL^q_x$ of (\ref{e:boussinesq equation}) with the same initial data.
\end{theo}

\begin{rema}
    Theorem \ref{maintheo} and Theorem \ref{maintheo2} represent the first results establishing the sharp nonuniqueness of the Boussinesq system, and the proof can also be adopted for the NS equation. However, the proof of Theorem \ref{maintheo} or Theorem \ref{maintheo2} is not valid in the case $\alpha\geq\frac{d+1}{2}$. More precisely, we can not balance the oscillation error which comes from convection term and the linear error (see Section \ref{errorsection}), which is the same obstacle encountered in \cite{qupeng}. 
\end{rema}

Due to the similar structure of NS equation, the Boussinesq system (\ref{e:boussinesq equation}) may possess the same weak-strong uniqueness properties in the critical or subcritical regime. More precisely, we have the following theorem, which corresponds to an end point $(\frac{2\alpha}{2\alpha-1},\infty)$ of generalized LPS criteria and reflect that non-uniqueness results of Theorem \ref{maintheo} and Theorem \ref{maintheo2} are sharp.
\begin{theo}\label{正面}
    Let $(u,\theta)$ be a weak solution of (\ref{e:boussinesq equation}) and $1\leq\alpha<\frac{d+2}{2}$ such that $(u,\theta)\in L^{\frac{2\alpha}{2\alpha-1}}_TL^\infty_x$. Then $(u,\theta)$ is a generalized Leray-Hopf weak solution and unique in the class of generalized Leray-Hopf weak solutions.
\end{theo}

We introduce the Boussinesq-Reynolds equation:
    \begin{equation}\label{e:Boussinesq-Reynold}
    \begin{cases}
    \partial_tv+\text{div}(v\otimes v)+\nabla p+(-\Delta)^{\alpha} v=\theta e_d+\text{div}R, \quad\quad \\
    \text{div}\,v=0,\\
    \partial_t\theta+\text{div}(v\theta)+(-\Delta)^{\alpha} \theta=\text{div}S,
    \end{cases}
    \end{equation}
where $R$ is a symmetric matrix function and $S$ is a vector function.

The following proposition serves as the basis for the proof of Theorem \ref{maintheo} or Theorem \ref{maintheo2}.
\begin{prop}\label{mainProp}
    Let $d\geq2$ be the dimension, $1\leq\alpha<\frac{d+1}{2}$, $1\leq p<\frac{2\alpha}{2\alpha-1}$, $q<\infty$ and $\epsilon>0$. For any smooth vector field $v\in C^\infty([0,T]\times\mathbb{T}^d;\mathbb{R}^d)$ and smooth function $\rho\in C^\infty([0,T]\times\mathbb{T}^d;\mathbb{R})$ with
    \begin{gather*}
        \int_{\mathbb{T}^d}v(t,x)dx=0,\,\text{div}\,v=0\quad and\quad\int_{\mathbb{T}^d}\rho(t,x)dx=0, 
    \end{gather*}
    there exist weak solutions $(u,\theta)$ of (\ref{e:boussinesq equation}) and a set 
    \begin{gather*}
        I=\cup^{\infty}_{i=1}(a_i,b_i)\subset[0,T]
    \end{gather*}
    such that the following holds.
    \begin{itemize}
        \item The weak solution $(u,\theta)$ belongs to $L^p_TL^\infty_x$ and satisfies
        \begin{gather*}
            \|u-v\|_{L^p_TL^\infty_x}+\|\theta-\rho\|_{L^p_TL^\infty_x}\leq\epsilon.
        \end{gather*}
        Moreover, if $1<\alpha<\frac{d+1}{2}$, we have
        \begin{gather*}
             \|u-v\|_{L^\frac{2\alpha}{2\alpha-1}_TL^q_x}+\|\theta-\rho\|_{L^\frac{2\alpha}{2\alpha-1}_TL^q_x}\leq\epsilon.
        \end{gather*}
        \item $(u,\theta)$ is smooth on $I\times\mathbb{T}^d$. Moreover, $(u,\theta)$ coincides with the unique smooth solution with the initial data $(v(0,x),\rho(0,x))$ and is also regular near $t=T$.
        \item The Hausdorff dimension of the residue set $S=[0,T]\backslash I$ satisfies
        \begin{gather*}
            d_{\mathcal{H}}(S)\leq\epsilon.
        \end{gather*}
    \end{itemize}
\end{prop}

The rest of the paper focuses on the proof of Proposition \ref{mainProp}. We will use the iteration scheme and there are two steps of construction. In Section \ref{section1}, we use the gluing procedure introduced by Isett in \cite{Isett3}, to concentrate the stress errors into many smaller sub-intervals. This step will make the singular set of the limit have a small Hausdorff dimension. In Section \ref{section of convex} and \ref{section3}, we add the suitable perturbation using the space-time convex integration method to reduce the size of the Reynolds stress error $(R,S)$. Finally, in Section \ref{section4}, we conclude the proof of Proposition \ref{mainProp}, Theorem \ref{maintheo}, Theorem \ref{maintheo2} and Theorem \ref{正面}.

\section{Concentration of the Reynolds stress error}\label{section1}
In this part, we employ gluing technique to concentrate the Reynolds stress errors. The support of the new Reynolds stress errors will be scaled down to our needed degree. At the same time, the new Reynolds stress errors will almost keep its $L^1_TL^r_x$ norm and the perturbation in velocity or in temperature will be arbitrary small in $L^\infty_TH^d_x$.

We introduce the notion of well-preparedness of solutions to (\ref{e:Boussinesq-Reynold}) from \cite{serrin准则luo}, aiming at ensuring that the singular set in time has a small Hausdorff dimension. 
\begin{defi}
    Let $\epsilon\in (0,1)$. We call smooth solution $(v,\theta,R,S)$ of (\ref{e:Boussinesq-Reynold}) are well prepared if there exist a set $I$ and a length scale $\tau>0$ such that $I$ is a union of at most $\tau^{-\epsilon}$ many closed intervals of length $5\tau$ and 
    \begin{gather*}
        R(t,x)=0\quad\text{and}\quad S(t,x)=0\quad\text{if}\quad\text{dist}(t,I^c)\leq\tau.
    \end{gather*}
\end{defi}

\begin{prop}\label{prop of gluing}
    Let $0<\epsilon<1$ and $(u,\theta,R,S)$ be a well-prepared smooth solution of (\ref{e:Boussinesq-Reynold}) of some set $I$ and length scale $\tau$. For any $1<r<\infty$ and $\delta>0$, there exists a universal constant $C(r,\epsilon,d)>0$ and a well-prepared smooth solution $(\bar{u},\bar{\theta},\bar{R},\bar{S})$ of some set $\bar{I}$ and length scale $\bar{\tau}$ such that the following holds.

    \noindent
        (1) The set $\bar{I}$ and the length sacle $\bar{\tau}$ satisfy 
        \begin{gather}
            \bar{I}\subset I \,\text{,}\,\{0,1\}\notin \bar{I}\,\text{and}\,\,\bar{\tau}<\frac{\tau}{2}.
        \end{gather}
        
    \noindent  
        (2) The new Reynolds stress error $\bar{R}$ and $\bar{S}$ satisfy 
        \begin{gather}
            \bar{R}(t,x)=0\quad\text{and}\quad \bar{S}(t,x)=0\quad if\quad\text{dist}(t,\bar{I}^c)\leq\frac{3\tau}{2},\label{suppbarR}\\
            \|\bar{R}\|_{L^1_TL^r_x}\leq C\|R\|_{L^1_TL^r_x}\,,\, \|\bar{S}\|_{L^1_TL^r_x}\leq C\|S\|_{L^1_TL^r_x}.
        \end{gather}
        
    \noindent
        (3)  The velocity perturbation $\bar{u}-u$ and the temperature perturbation $\bar{\theta}-\theta$ satisfy
        \begin{gather}
            \text{supp}(\bar{u}-u)\subset I\times\mathbb{T}^d\,,\,\text{supp}(\bar{\theta}-\theta)\subset I\times\mathbb{T}^d,\\\|\bar{u}-u\|_{L^\infty_TH^d_x}\leq\delta\,,\,\|\bar{\theta}-\theta\|_{L^\infty_TH^d_x}\leq\delta.
        \end{gather}

\end{prop}

\subsection{Gluing procedure}

Firstly, we construct $(\bar{u},\bar{\theta},\bar{R},\bar{S})$. Let $\bar{\tau}>0$ and $\epsilon\in(0,1)$, we define
\begin{gather*}
    t_i=i\bar{\tau}^\epsilon,\quad\forall i\in [0,\frac{T}{\bar{\tau}^\epsilon}].
\end{gather*}
Here we assume $\frac{T}{\bar{\tau}^\epsilon}$ is always an integer without loss of generality.
For any $i\in [0,\frac{T}{\bar{\tau}^\epsilon}-1]$, let $v_i:[t_i,t_{i+1}]\xrightarrow{}\mathbb{T}^d$ and $\phi_i:[t_i,t_{i+1}]\xrightarrow{}\mathbb{T}$ be the smooth solutions of the following Boussinesq system
\begin{equation}\label{e:gluing}
    \begin{cases}
    \partial_tv_i+\text{div}(v_i\otimes u)+\text{div}(u\otimes v_i)+\text{div}(v_i\otimes v_i)+\nabla p_i+(-\Delta)^{\alpha} v_i=\phi _ie_d-\text{div}R, \quad\quad \\
    \text{div}\,v_i=0,\\
    \partial_t\phi_i+\text{div}(u\phi_i)+\text{div}(v_i\theta)+\text{div}(v_i\phi_i)+(-\Delta)^{\alpha} \phi_i=-\text{div}S,\\
    v_i(t_i,x)=\phi_i(t_i,x)=0.
    \end{cases}
\end{equation}

The next step involves gluing $\{v_i\}$ or $\{\phi_i\}$ together to form a new solution: this new solution is the exact solution of the Boussinesq equation on the majority of $[0,T]$, with the error being supported on multiple disjoint small subintervals. We define smooth cutoff functions $\{\chi_i:[0,T]\rightarrow[0,1]\}$ as follows.

For any $i\in [1,\frac{T}{\bar{\tau}^\epsilon}-2]$,
\begin{gather*}
    \chi_i(t)\triangleq\begin{cases}
        1\quad\text{if}\,t_i+\bar{\tau}\leq t\leq t_{i+1}-\bar{\tau},
        \\0\quad\text{if}\,t\leq t_i+\frac{\bar{\tau}}{2}\,\text{or}\,t\geq t_{i+1}-\frac{\bar{\tau}}{2},
    \end{cases}
\end{gather*}
and
\begin{gather*}
    \chi_0(t)\triangleq\begin{cases}
        1\quad\text{if}\,0\leq t\leq t_{1}-\bar{\tau},
        \\0\quad\text{if}\,t\geq t_{1}-\frac{\bar{\tau}}{2},
    \end{cases}
\end{gather*}
and
\begin{gather*}
    \chi_{\frac{T}{\tau^\epsilon}-1}(t)\triangleq\begin{cases}
        1\quad\text{if}\,t_{\frac{T}{\tau^\epsilon}-1}+\bar{\tau}\leq t\leq1,
        \\0\quad\text{if}\,t\leq t_{\frac{T}{\tau^\epsilon}-1}+\frac{\bar{\tau}}{2}.
    \end{cases}
\end{gather*}

We define the new velocity 
\begin{gather*}
\bar{u}\triangleq u+\sum_i\chi_iv_i,\label{defi of baru}
\end{gather*}
and the new temperature
\begin{gather*}
\bar{\theta}\triangleq\theta+\sum_i\chi_i\phi_i.\label{defi of bartheta}
\end{gather*}
Using (\ref{e:gluing}), we have
\begin{align*}
    \partial_t\bar{u}-\Delta^{\alpha}\bar{u}+\text{div}(\bar{u}\otimes\bar{u})+\nabla p =&\theta e_d+\text{div}R+(\partial_t-\Delta^{\alpha})(\sum_i\chi_iv_i)\nonumber\\ &+\sum_i\chi_i\left(\text{div}(u\otimes v_i)+\text{div}(v_i\otimes u)\right)+\sum_i\chi_i^2\text{div}(v_i\otimes v_i)\nonumber\\
    =&\theta e_d+\text{div}R+\sum_i\partial_t\chi_i\cdot v_i\nonumber\\&+\sum_i({\chi_i}^2-\chi_i)\text{div}(v_i\otimes v_i)+\sum_i\chi_i(-\text{div}R+\theta_ie_d)-\sum_i\chi_i\nabla q_i.\label{middle equation}
\end{align*}
Thus, we introduce new concentrated Reynolds stress and new pressure as
\begin{gather*}
    \bar{R}\triangleq(1-\sum_i\chi_i)R+\mathcal{R}(\sum_i\partial_t\chi_i\cdot v_i)+\sum_i(\chi_i^2-\chi_i)v_i\otimes v_i,
    \\ \bar{p}\triangleq p+\sum
    _i\chi_iq_i.
\end{gather*}
Here $\mathcal{R}:C^\infty(\mathbb{T}^d;\mathbb{R}^d)\rightarrow C^\infty(\mathbb{T}^d;S^{d\times d})$ is an antidivergence operator (see \ref{def of antidiv}). 
Therefore, we have 
\begin{gather*}
\begin{cases}
    \partial_t\bar{u}+(-\Delta)^{\alpha}\bar{u}+\text{div}(\bar{u}\otimes\bar{u})+\nabla \bar{p}=\bar{\theta}e_d+\text{div}\bar{R},\\
    \text{div}\,\bar{u}=0.
\end{cases}
\end{gather*}

Following the same procedure, we define
\begin{gather*}
    \bar{S}\triangleq(1-\sum_i\chi_i)S+\mathcal{G}(\sum_i\partial_t\chi_i\cdot\theta_i)+\sum_i(\chi_i^2-\chi_i)v_i\cdot\theta_i.
\end{gather*}
Here $\mathcal{G}:C^\infty(\mathbb{T}^d;\mathbb{R})\rightarrow C^\infty(\mathbb{T}^d;\mathbb{R}^d)$ is another type of antidivergence operator (see \ref{def of antidiv}). Finally, $(\bar{u},\bar{\theta},\bar{R},\bar{S})$ satisfy the Boussinesq-Renolds equation:
\begin{gather*}
    \begin{cases}
        \partial_t\bar{u}+\text{div}(\bar{u}\otimes \bar{u})+\nabla \bar{p}+(-\Delta)^{\alpha} \bar{u}=\bar{\theta} e_d+\text{div}\bar{R}, \quad\quad \\
    \text{div}\,\bar{u}=0,\\
    \partial_t\bar{\theta}+\text{div}(\bar{u}\bar{\theta})+(-\Delta)^{\alpha} \bar{\theta}=\text{div}\bar{S}. 
    \end{cases}
\end{gather*}

\subsection{Proof of Proposition \ref{prop of gluing}}
In the above subsection, we define $(\bar{u},\bar{\theta},\bar{R},\bar{S})$. In this part, we prove the proposition \ref{prop of gluing}. First, we present an additional proposition that quantifies the corrector $v_i$ and $\phi_i$ with respect to the time scale $\bar{\tau}$ as well as the forcings $-\text{div}R$ and $-\text{div}S$.
\begin{prop}\label{prop of vi}
    Let $d\geq2$, $(u,\theta,R,S)$ be smooth functions, $1<r<\infty$ and $\delta>0$. There exists a universal constant $C_r$ depending only on $r$ and dimension $d$ such that the following holds.
    
    If $\bar{\tau}$ is sufficiently small, then the smooth solutions $(v_i,\phi_i)$ to (\ref{e:gluing}) on $[t_i,t_{i+1}]$ satisfy 
    \begin{gather}
        \|(v_i,\phi_i)\|_{L^\infty(t_i,t_{i+1};H^d_x(\mathbb{T}^d))}\leq\delta,\label{smallestimate}\\
        \|\mathcal{R}v_i\|_{L^\infty(t_i,t_{i+1};L^r(\mathbb{T}^d))}\leq C_r\|R\|_{L^1(t_i,t_{i+1};L^r(\mathbb{T}^d))}+C_{r,u,\theta}\delta\bar{\tau}^\epsilon,\label{lp of Rv}\\
        \|\mathcal{G}\phi_i\|_{L^\infty(t_i,t_{i+1};L^r(\mathbb{T}^d))}\leq C_r\|S\|_{L^1(t_i,t_{i+1};L^r(\mathbb{T}^d))}+C_{r,u,\theta}\delta\bar{\tau}^\epsilon,\label{lp of Gphi}
    \end{gather}
    here the constant $C_{r,u,\theta}$ depends only on $r$, $u$ and $\theta$.
\end{prop}
\begin{proof}
    Since $\bar{\tau}$ is sufficiently small, (\ref{smallestimate}) is obtained directly by using the standard energy estimate and continuity method. In the following, we focus on the demonstration of (\ref{lp of Rv}).
    
    Denote by $\mathbb{P}_H$ the Leray projection operator. Applying antidivergence operator $\mathcal{R}$ to both sides of the first equation in (\ref{e:gluing}), we have
    \begin{gather*}
        (\partial_t+(-\Delta)^{\alpha})\mathcal{R}v_i=-\mathcal{R}\mathbb{P}_H\text{div}(v_i\otimes u+u\otimes v_i+v_i\otimes v_i)-\mathcal{R}\mathbb{P}_H\text{div}R+\mathcal{R}\mathbb{P}_H(\phi_ie_d)
    \end{gather*}

    By (\ref{smallestimate}) and the fact that $\mathcal{R}\mathbb{P}_H\text{div}$ is a Calderon-Zygmund operator on $L^r(\mathbb{T}^d)$ and $t_{i+1}-t_i=\bar{\tau}^\epsilon$, we obtain
    \begin{align}\label{shizi1}
        &\|\mathcal{R}\mathbb{P}_H\text{div}(v_i\otimes u+u\otimes v_i+v_i\otimes v_i)+\mathcal{R}\mathbb{P}_H\text{div}R\|_{L^1(t_i,t_{i+1};L^r(\mathbb{T}^d))}\nonumber\\
        \leq &C_r(\|v_iu\|_{L^1(t_i,t_{i+1};L^r(\mathbb{T}^d))}+\|v_i^2\|_{L^1(t_i,t_{i+1};L^r(\mathbb{T}^d))}+\|R\|_{L^1(t_i,t_{i+1};L^r(\mathbb{T}^d))})\nonumber\\
        \leq& C_r(\bar{\tau}^\epsilon\|u\|_{L^\infty_{t,x}}\|v_i\|_{L^\infty_{t,x}}+\bar{\tau}^\epsilon\|v_i\|_{L^\infty_{t,x}}^2+\|R\|_{L^1(t_i,t_{i+1};L^r(\mathbb{T}^d))})\\
        \leq & C_r\|R\|_{L^1(t_i,t_{i+1};L^r(\mathbb{T}^d))}+C_{r,u,\theta}\delta\bar{\tau}^\epsilon.
    \end{align}

    Due to Theorem \ref{theo of antidiv} and the fact that $\phi_i$ has zero mean, we obtain
    \begin{gather}\label{shizi2}
        \|\mathcal{R}\mathbb{P}_H(\phi_ie_d)\|_{L^1(t_i,t_{i+1};L^r(\mathbb{T}^d))}\leq C_r\|\phi_i\|_{L^1(t_i,t_{i+1};L^r(\mathbb{T}^d))}\leq C_{r,u,\theta}\delta\bar{\tau}^\epsilon.
    \end{gather}

    Combining (\ref{shizi1}), (\ref{shizi2}) and property of the Heat equation, it follows that
    \begin{align*}
        &\|\mathcal{R}v_i\|_{L^\infty(t_i,t_{i+1};L^r(\mathbb{T}^d))}\\
        \leq &C_r\|-\mathcal{R}\mathbb{P}_H\text{div}(v_i\otimes u+u\otimes v_i+v_i\otimes v_i)-\mathcal{R}\mathbb{P}_H\text{div}R+\mathcal{R}\mathbb{P}_H(\phi_ie_d)\|_{L^1(t_i,t_{i+1};L^r(\mathbb{T}^d))}\\
        \leq &C_r\|R\|_{L^1(t_i,t_{i+1};L^r(\mathbb{T}^d))}+C_{r,u,\theta}\delta\bar{\tau}^\epsilon,
    \end{align*}
which gives (\ref{lp of Rv}). The proof of (\ref{lp of Gphi}) is similar, here we omit the detail.
\end{proof}

Finally, we complete the proof of Proposition \ref{prop of gluing}.
\begin{proof}[\textbf{Proof of Proposition \ref{prop of gluing}}]
    Let $\bar{\tau}$ be sufficiently small such that $10\bar{\tau}^\epsilon<\tau$ and $\bar{\tau}<\frac{\tau}{2}$. We define the set on time axis $\bar{I}$ in Proposition \ref{prop of gluing} as follows,
    \begin{gather*}
        \bar{I}\triangleq\bigcup_{i\in\Sigma}\left[t_i-\frac{5\bar{\tau}}{2},t_i+\frac{5\bar{\tau}}{2}\right].
    \end{gather*}
    where $\Sigma=\{i\in\mathbb{Z}:1\leq i\leq \frac{T}{\bar{\tau}^\epsilon}\,\,\text{and}\,\, (R,S)\neq 0\,\text{on}\,[t_{i-1},t_i+\frac{5\bar{\tau}}{2}]\cap[0,T]\}$.
    Thus the temporal support and the well-preparedness of $(\bar{u},\bar{\theta},\bar{R},\bar{S})$ follow easily from the definition of $\{\chi_i\}$ and $(\bar{R},\bar{S})$. We only need to verify $\bar{I}\subset I$ and we argue by contradiction.

    If there exists $a\in \bar{I}\cap I^c$, then there exists $i\in \Sigma$ such that
    \begin{gather*}
     a\in[t_i-\frac{5\bar{\tau}}{2},t_i+\frac{5\bar{\tau}}{2}]\subset[t_{i-1},t_i+\frac{5\bar{\tau}}{2}],\\
        (R,S)|_{[t_{i-1},t_i+\frac{5\bar{\tau}}{2}]}\neq0,\\
        (R,S)|_{(a-\tau,a+\tau)}=0,
    \end{gather*}
    which contracts the fact that $[t_{i-1},t_i+\frac{5\bar{\tau}}{2}]\subset(a-\tau,a+\tau)$ due to $10\bar{\tau}^\epsilon<\tau$.
    
    In the following, we focus on the estimates of $(\bar{u},\bar{\theta},\bar{R},\bar{S})$. Due to the definition of $\bar{R}$ and the triangle inequality, we get
\begin{align*}
    \|\bar{R}\|_{L^1_TL^r_x}\leq\|1-\sum_i\chi_i\|_{L^\infty_T}\|R\|_{L^1_TL^r_x}+\sum_i\|\partial_t\chi_i\|_{L^1_T}\|\mathcal{R}v_i\|_{L^\infty_TL^r_x}+\sum_i\|\chi_i-\chi_i^2\|_{L^1_T}\|v_i\otimes v_i\|_{L^\infty_TL^r_x}.
\end{align*}
By the definition of $\{\chi_i\}$, we have the following trivial bounds in time:
\begin{gather}
    \|1-\sum_i\chi_i\|_{L^\infty_T}\leq1,\label{901}
\end{gather}
and for any $i\in[0,\frac{T}{\bar{\tau}^\epsilon}-1]$,
\begin{gather}
    \|\partial_t\chi_i\|_{L^1_T}\lesssim1,\label{902}\\
    \|\chi_i-\chi_i^2\|_{L^1_T}\lesssim\bar{\tau}\label{903}.
\end{gather}
Combing (\ref{901}), (\ref{902}), (\ref{903}) and Proposition \ref{prop of vi}, we obtain
\begin{align*}
    \|\bar{R}\|_{L^1_TL^r_x}\lesssim\|R\|_{L^1_TL^r_x}+\sum_i(C_r\|R\|_{L^1(t_i,t_{i+1};L^r(\mathbb{T}^d))}+C_{r,u,\theta}\delta\bar{\tau}^\epsilon)+\sum_i\delta^2\bar{\tau}.
\end{align*}
Using the fact that $\sum_i\bar{\tau}\leq\sum_i\bar{\tau}^\epsilon\leq T$ and choosing $\bar{\tau}$ sufficiently small, we get
\begin{gather*}
    \|\bar{R}\|_{L^1_TL^r_x}\leq C\|R\|_{L^1_TL^r_x}.
\end{gather*}

By the definition of $\{\chi_i\}$, (\ref{defi of baru}) and Proposition \ref{prop of vi}, we have
\begin{gather*}
    \|\bar{u}-u\|_{L^\infty_TH^d_x}\leq\|\sum_i\chi_iv_i\|_{L^\infty_TH^d_x}\leq \delta.
\end{gather*}

Using the same method, we can obtain the estimate of $(\bar{\theta},\bar{S})$, the proof is omitted herein for brevity.

\end{proof}

\section{Convex integration in space-time}\label{section of convex}
This section concentrates on reducing the size of the Reynolds stress by using the convex integration method. The main step is to construct a suitable velocity perturbation $\omega$ and a suitable temperature perturbation $\kappa$ such that
\begin{gather*}
    \tilde{u}=\bar{u}+\omega,\quad\tilde{\theta}=\bar{\theta}+\kappa
\end{gather*}
solve the Boussinesq-Reynolds equation with smaller Reynolds stress $\tilde{R}$ and $\tilde{S}$. Here we present the main proposition of this section.

\begin{prop}\label{convexi}
    Let $d\geq2$ be the dimension, $1\leq\alpha<\frac{d+1}{2}$, $1\leq p<\frac{2\alpha}{2\alpha-1}$ and $q<\infty$, there exists $r>1$ such that we have the following.

    Let $\delta>0$ and $(\bar{u},\bar{\theta},\bar{R},\bar{S})$ be a well-prepared smooth solution of (\ref{e:Boussinesq-Reynold}) for the set $\bar{I}$ and the length scale $\bar{\tau}$, there exists another well-prepared smooth solution $(\tilde{u},\tilde{R},\tilde{\theta},\tilde{S})$ of (\ref{e:Boussinesq-Reynold}) for the same set $\bar{I}$ and the same length scale $\bar{\tau}$ such that the following holds.

    \noindent
    (1) The new Reynolds stress error $\tilde{R}$ and $\tilde{S}$ satisfies
    \begin{gather*}
        \|(\tilde{R},\tilde{S})\|_{L^1_TL^r_x}\leq\delta.
    \end{gather*}
    \noindent
    (2) The velocity perturbation $\omega\triangleq\tilde{u}-\bar{u}$ and the temperature perturbation $\kappa\triangleq\tilde{\theta}-\bar{\theta}$ satisfy
    \begin{gather*}
        \text{supp}_t\,(\omega,\kappa)\subset \bar{I},\\
        \|(\omega,\kappa)\|_{L^2([0,T]\times\mathbb{T^d})}\leq M\|(\bar{R},\bar{S})\|_{L^1([0,T]\times\mathbb{T^d})},\\
        \|(\omega,\kappa)\|_{{L^p_TL^\infty_x}}\leq\delta.
    \end{gather*}    
     Moreover, if $1<\alpha<\frac{d+1}{2}$, we have   
        \begin{gather*}
        \|(\omega,\kappa)\|_{{L^\frac{2\alpha}{2\alpha-1}_TL^q_x}}\leq\delta.
    \end{gather*}
\end{prop}

\subsection{Concentrated Mikado flows}
Firstly, we recall the stationary Mikado flows that are introduced in \cite{6}. They serve as the key spatial building blocks in the convex integration method. What we use are called concentrated Mikado flows that can be found in \cite[Theorem 4.3]{serrin准则luo} or \cite[Section 4.1]{qupeng}. Here we recall $e_k=\frac{k}{|k|}$ for any $k\in\Lambda$.
\begin{prop}\label{mikado}
    Let $d\geq2$ be the dimension and $\Lambda\subset\mathbb{Z}^d$ be the finite set given by Lemma \ref{geometric lemma}. For any $k\in\Lambda$ and $\mu>0$, there exist two smooth functions $\Psi_k^\mu:\mathbb{T}^d\rightarrow\mathbb{R}$ and $\Phi_k^\mu:\mathbb{T}^d\rightarrow\mathbb{R}$ such that the following hold.
    
    \noindent
    (1) For any $k\in\Lambda$, $\Psi_k^\mu$ and $\Phi_k^\mu$ satisfy
    \begin{gather*}
         \Psi_k^\mu=\Delta\Phi_k^\mu.
    \end{gather*}
    
    \noindent
    (2) For any $k\in\Lambda$, $\Psi_k^\mu e_k$ is divergence-free, solves the stationary Euler equations without pressure
    \begin{gather}
        \text{div}(\Psi_k^\mu e_k\otimes\Psi_k^\mu e_k)=0,\label{divpsi=0}
    \end{gather}
    and can be written as a divergence of a skew-symmetric tensor
   \begin{gather}\label{Phi=divomega}
       \Psi_k^\mu e_k=\text{div}\,\Omega_k^\mu,
   \end{gather}
    which $\Omega_k^\mu$ is defined as follows,
   \begin{gather*}
       \Omega_k^\mu\triangleq e_k\otimes\nabla\Phi_k^\mu-\nabla\Phi_k^\mu\otimes e_k
   \end{gather*}

   \noindent
   (3) For any $k\in\Lambda$, $1\leq p\leq \infty$ and $m\geq0$, there holds
   \begin{gather*}
       \mu^{-m}\|\nabla^m\Psi_k^\mu\|_{L^p(\mathbb{T}^d)}\lesssim_m\mu^{\frac{d-1}{2}-\frac{d-1}{p}},\\
       \mu^{-m}\|\nabla^{m}\Phi_k^\mu\|_{L^p(\mathbb{T}^d)}\lesssim_m\mu^{-2+\frac{d-1}{2}-\frac{d-1}{p}},\\
       \fint_{\mathbb{T}^d}(\Psi_k^\mu)^2=1.
   \end{gather*}

   \noindent
   (4) For any $k\neq k'\in\Lambda$ and $1\leq p\leq \infty$, there holds
   \begin{gather*}
       \|\Psi_k^\mu\Psi_{k'}^\mu\|_{L^p(\mathbb{T}^d)}\lesssim\mu^{d-1-\frac{d}{p}}.
   \end{gather*}
   In the above estimates, the implicit constants are uniform with respect to $\mu$.
\end{prop}

Denote 
\begin{gather*}
    W_k^\mu\triangleq\Psi_k^\mu e_k.
\end{gather*}
$\{W_k^\mu\}_k$ are called the concentrated Mikado flows and $\mu$ is  spatial concentration parameter.

\subsection{Temporal intermittency and space-time cutoffs}
For obtaining estimates in $L^p_TL^\infty_x$ or $L^{\frac{2\alpha}{2\alpha-1}}_TL_x^q$ estimates, implementing temporal concentration during the construction of perturbation is a necessary step.

Let $g\in C^\infty_c((0,T))$ and $l>0$ be a large constant which will be determined later. We define a T-periodic function $g_l:\mathbb{R}\rightarrow\mathbb{R}$ that satisfies
\begin{gather*}
    g_l(t)=l^{\frac{1}{2}}g(lt),\quad\forall t\in[0,T].
\end{gather*}

In order to cancel the high temporal frequency due to the temporal concentration, we define the following temporal corrector.
\begin{gather*}
    h_l(t)\triangleq\int^t_0g_l^2(s)-1\,ds,\quad\forall t\in[0,T].
\end{gather*}

From the above definition, we can easily obtain the following estimate.
\begin{prop}\label{prop of gl}
    For any $p\in[1,\infty]$ and $\nu,m\in \mathbb{N}$, we have
    \begin{align}
        \|\partial_t^m\left(g_l(\nu\cdot)\right)\|_{L^p_T}&\lesssim (vl)^m l^{\frac{1}{2}-\frac{1}{p}},\\
        \|g_l(\nu\cdot)\|_{L^2_T}&=1,\\
        \|h_l\|_{L^\infty_T}&\leq1.        
    \end{align}
    Here the implicit constants in above estimates hold uniformly in temporal concentration parameter $l$ and temporal oscillation parameter $\nu$.
\end{prop}

In the following, we define two cutoff functions. The first is to meet the hypotheses in Lemma \ref{geometric lemma} and the second is to ensure the well-preparedness of $(\tilde{u},\tilde{R},\tilde{\theta},\tilde{S})$. 

We define a smooth, monotonically increasing function $\xi:[0,\infty)\rightarrow[0,\infty)$ that satisfies
\begin{gather*}
    \xi(x)=\begin{cases}
        \frac{2\|(\bar{R},\bar{S})\|_{L^1([0,T]\times\mathbb{T^d})}}{r_o}\quad&\text{if}\,0\leq|x|\leq\|(\bar{R},\bar{S})\|_{L^1([0,T]\times\mathbb{T^d})},\\
        \frac{2|x|}{r_0}\quad&\text{if}\,|x|\geq2\|(\bar{R},\bar{S})\|_{L^1([0,T]\times\mathbb{T^d})}.
    \end{cases}
\end{gather*}
Here $r_0$ is the constant given in Lemma \ref{geometric lemma}.
We denote 
\begin{gather*}
    \rho(t,x)\triangleq\xi(|(\bar{R},\bar{S})(t,x)|),\quad\forall(t,x)\in [0,T]\times\mathbb{T}^d.
\end{gather*}
Then we immediately obtain
\begin{gather*}
    \text{Id}-\frac{\bar{R}(t,x)}{\rho(t,x)}\in B_{r_0}(\text{Id}),\quad\forall(t,x)\in[0,T]\times\mathbb{T}^d,
\end{gather*}
and according to Lemma \ref{geometric lemma},
\begin{gather*}
    \Gamma_k(\text{Id}-\frac{\bar{R}(t,x)}{\rho(t,x)})
\end{gather*}
is well defined.

To guarantee that $(\tilde{u},\tilde{R},\tilde{\theta},\tilde{S})$ is well-prepared for the set $\bar{I}$ and the length scale $\bar{\tau}$, we define a smooth temporal function $f:[0,T]\rightarrow[0,1]$ that satisfies
\begin{gather*}
    f(t)=\begin{cases}
        1\quad\text{if}\,\text{dist}(t,\bar{I}^c)\geq\frac{3\bar{\tau}}{2},\\
        0\quad\text{if}\,\text{dist}(t,\bar{I}^c)\leq\ \bar{\tau}.
    \end{cases}
\end{gather*}

\subsection{Construction of the new perturbation and the new Reynolds stress error}
We recall four parameters for the perturbation:

\noindent
(1) Temporal oscillation $\nu\in\mathbb{N}$ and temporal concentration $\kappa>0$.

\noindent
(2) Spatial oscillation $\sigma\in\mathbb{N}$ and spatial concentration $\mu\in\mathbb{N}$.

We define two amplitude functions:
\begin{gather*}
    a_k(t,x)\triangleq g_l(\nu t)f(t)\rho(t,x)^{\frac{1}{2}}\Gamma_k(\text{Id}-\frac{\bar{R}(t,x)}{\rho(t,x)}),\\
    b_k(t,x)\triangleq g_l(\nu t)f(t)\rho(t,x)^{\frac{1}{2}}\gamma_k(-\frac{\bar{S}(t,x)}{\rho(t,x)})\left[\Gamma_k(\text{Id}-\frac{\bar{R}(t,x)}{\rho(t,x)})\right]^{-1},
\end{gather*}
where $\Gamma_k$ and $\gamma_k$ are given in Lemma \ref{geometric lemma}. Since  $\Gamma$ is bounded below by a positive constant $c_0>0$ from Lemma \ref{geometric lemma}, $b_k$ is well defined.

The principle part of the new velocity perturbation $\omega^{(p)}$ and the new temperature perturbation $\kappa^{(p)}$ are composed of the amplitude function and the concentrated Mikado flow with oscillation:
\begin{gather}
    \omega^{(p)}(t,x)\triangleq\sum_{k\in\Lambda}a_k(t,x)\Psi_k^\mu (\sigma x)e_k,\label{DEFofWP}\\
    \kappa^{(p)}(t,x)\triangleq\sum_{k\in\Lambda}b_k(t,x)\Psi_k^\mu (\sigma x).\label{DEFofKP}
\end{gather}

Note also that $\omega^{(p)}$ is not divergence-free. To rectify this, we introduce the divergence-free corrector:
\begin{gather}
    \omega^{(c)}(t,x)\triangleq\sum_{k\in\Lambda} \sigma^{-1}\nabla a_k(t,x):\Omega^\mu_k(\sigma x).\label{DEFofWC}
\end{gather}
Thus, by (\ref{Phi=divomega}), 
\begin{align*}
 \omega^{(p)}+\omega^{(c)}&=\sigma^{-1}\sum_{k\in\Lambda}a_k(t,x)\text{div}\,\Omega^\mu_k(\sigma x)+\sigma^{-1}\sum_{k\in\Lambda} \nabla a_k(t,x):\Omega^\mu_k(\sigma x)\\
   &=\sigma^{-1}\text{div}(\sum_{k\in\Lambda}a_k(t,x)\Omega^\mu_k(\sigma x)).
\end{align*}
Since $a_k(t,x)\Omega^\mu_k$ is skew-symmetric, $\omega^{(p)}+\omega^{(c)}$ is manifestly divergence-free.

To achieve a balance for the high temporal frequency component in the interaction, temporal correctors are introduced herein.
\begin{gather}
     \omega^{(t)}(t,x)\triangleq\nu^{-1}h_l(\nu t)\mathbb{P}_H\text{div}\bar{R},\label{DEFofWT}\\
     \kappa^{(t)}(t,x)\triangleq \nu^{-1}h_l(\nu t)\text{div}\bar{S},\label{DEFofKT}
\end{gather}
$ \omega^{(t)}$ is obviously divergence-free.

In summary, by the above definition, we define the total velocity perturbation and the total temperature perturbation:
\begin{gather*}
    \omega\triangleq \omega^{(p)}+\omega^{(c)}+\omega^{(t)},\\
    \kappa\triangleq \mathbb{P}_{\neq0}\kappa^{(p)}+\kappa^{(t)},
\end{gather*}
where $\mathbb{P}_{\neq0}$ denotes the projection onto non-zero Fourier modes. Here we 
note that $\omega$ is divergence-free and $\kappa$ has zero mean.

Having defined the perturbation in the prior part, we focus on the construction of the new Reynolds stress. First, we carry out the computation of the interactions associated with the principal part:
\begin{gather}
    \text{div}(\omega^{(p)}\otimes\omega^{(p)}+\bar{R})
    =\text{div}\left[\sum_{k\in\Lambda}(a_k)^2(\Psi_k^\mu )^2e_k\otimes e_k+\bar{R}\right]+\text{div}R_{far},\label{divRfar}\\
    \text{div}(\omega^{(p)}\kappa^{(p)}+\bar{S})=\text{div}\left[ \sum_{k\in\Lambda}a_kb_k(\Psi_k^\mu )^2e_k+\bar{S} \right]+\text{div}S_{far},\label{divSfar}
\end{gather}
where $R_{far}$ and $S_{far}$ denote the nonlocal interactions between Mikado flows with different directions and are given by
\begin{gather}
    R_{far}\triangleq\sum_{k\neq k'}\mathcal{R}\text{div}\left(a_ka_{k'}\Psi_k^\mu (\sigma x)\Psi_{k'}^\mu (\sigma x)e_k\otimes e_{k'}\right),\label{DEFofRfar}\\
     S_{far}\triangleq\sum_{k\neq k'}a_kb_{k'}\Psi_k^\mu (\sigma x)\Psi_{k'}^\mu (\sigma x)e_k.\label{DEFofSfar}
\end{gather}

By a direct computation and Lemma \ref{geometric lemma}, we obtain
\begin{align*}
    \text{div}(\sum_{k\in\Lambda}a_k^2e_k\otimes e_k+\bar{R})&=\text{div}\left[f^2g_l^2(\nu t)\rho\sum_{k\in\Lambda}\Gamma_k^2(\text{Id}-\frac{\bar{R}}{\rho})e_k\otimes e_k+\bar{R}\right]\\
    &=\text{div}\left[f^2g_l^2(\nu t)\rho\text{Id}-f^2g_l^2\bar{R}+\bar{R}\right]\\
    &=\nabla (f^2g_l^2(\nu t)\rho)+\text{div}\left[(1-g_l^2(\nu t))\bar{R}\right],
\end{align*}
and
\begin{align*}
    \text{div}(\sum_{k\in\Lambda}a_kb_ke_k\otimes e_k+\bar{S})&=\text{div}\left[f^2g_l^2(\nu t)\rho\sum_{k\in\Lambda}\gamma_k(-\frac{\bar{S}}{\rho})e_k+\bar{S}\right]\\
    &=\text{div}\left[(1-f^2g_l^2(\nu t))\bar{S}\right]\\
     &=\text{div}\left[(1-g_l^2(\nu t))\bar{S}\right],
\end{align*}
where the last equality in both is valid due to (\ref{suppbarR}) and the definition of $f$. Using the definition of $\omega^{(t)}$ and $\kappa^{(t)}$, we have
\begin{gather}
     \text{div}(\sum_{k\in\Lambda}a_k^2e_k\otimes e_k+\bar{R})=\nabla (f^2g_l^2(\nu t)\rho)-\partial_t \omega^{(t)}+\text{div}R_{osc,t}-\nu^{-1}\nabla\Delta^{-1}\text{div}\text{div}(\partial_t\bar{R}h_l(\nu t)),\label{divRosct}\\
     \text{div}(\sum_{k\in\Lambda}a_kb_ke_k\otimes e_k+\bar{S})=-\partial_t\kappa^{(t)}+\text{div}S_{osc,t},\label{divSosct}
\end{gather}
where $R_{osc,t}$ and $S_{osc,t}$ denote temporal oscillation error and are given by
\begin{gather}
     R_{osc,t}\triangleq\nu^{-1}h_l(\nu t)\partial_t\bar{R},\label{DEFofRosct}\\
     S_{osc,t}\triangleq\nu^{-1}h_l(\nu t)\partial_t\bar{S}\label{DEFofSosct}.
\end{gather}

Applying (\ref{divpsi=0}) and the two bilinear operators introduced in Proposition \ref{bilinear B}, we deduce 
\begin{align}
    \text{div}\sum_{k\in\Lambda}a_k^2((\Psi_k^\mu )^2-1)e_k\otimes e_k&=\sum_{k\in\Lambda}\nabla(a_k^2)\cdot((\Psi_k^\mu )^2-1)e_k\otimes e_k\label{divRoscx}\\
    &=\text{div}\sum_{k\in\Lambda} \mathcal{B}(\nabla(a_k^2),\mathbb{P}_{\neq0}(\Psi_k^\mu (\sigma x)e_k\otimes\Psi_k^\mu (\sigma x)e_k))\nonumber
\end{align}
and
\begin{align}
    \text{div}\sum_{k\in\Lambda}a_kb_k((\Psi_k^\mu )^2-1)e_k&=\sum_{k\in\Lambda}e_k\cdot\nabla(a_kb_k)(\Psi_k^\mu(\sigma x)^2-1)\label{divSoscx}\\
    &=\text{div}\sum_{k\in\Lambda}\tilde{\mathcal{B}}(e_k\nabla(a_kb_k),\mathbb{P}_{\neq0}\Psi_k^\mu (\sigma x)^2).\nonumber
\end{align}
Here we note that (\ref{divRoscx}) and (\ref{divSoscx}) have zero spatial mean. Thus, we define two spatial oscillation errors
\begin{gather}
    R_{osc,x}\triangleq\mathcal{B}(\nabla(a_k^2),\mathbb{P}_{\neq0}(\Psi_k^\mu (\sigma x)^2e_k\otimes e_k)),\label{DEFofRoscx}\\
    S_{osc,x}\triangleq \tilde{\mathcal{B}}(e_k\nabla(a_kb_k),\mathbb{P}_{\neq0}\Psi_k^\mu (\sigma x)^2).\label{DEFofSoscx}
\end{gather}
Combining (\ref{divRfar}), (\ref{divSfar}), (\ref{divRosct}), (\ref{divSosct}), (\ref{divRoscx}) and (\ref{divSoscx}), we obtain
\begin{align}
    \text{div}(\omega^{(p)}\otimes\omega^{(p)}+\bar{R})+\partial_t\omega^{(t)}=&\nabla (f^2g_l^2(\nu t)\rho)-\nu^{-1}\nabla\Delta^{-1}\text{div}\text{div}(\partial_t\bar{R}h_l(\nu t))\\&+\text{div}R_{far}+\text{div}R_{osc,t}+\text{div}R_{osc,x},\nonumber\\
    \text{div}(\omega^{(p)}\kappa^{(p)}+\bar{S})+\partial_t\kappa^{(t)}=&\text{div}S_{far}+\text{div}S_{osc,t}+\text{div}S_{osc,x}.\label{666}
\end{align}

As usual, we define the correction error and the linear error:
\begin{align}
    R_{cor}&\triangleq\mathcal{R}\,\text{div}((\omega^{c}+\omega^{(t)})\otimes\omega+\omega^{(p)}\otimes(\omega^{c}+\omega^{(t)})),\label{DEFofRcor}\\
    R_{lin}&\triangleq\mathcal{R}(\partial_t(\omega^{p}+\omega^{(c)})+(-\Delta)^{\alpha} \omega+\text{div}(\bar{u}\otimes\omega+\omega\otimes\bar{u})-\kappa e_d),\label{DEFofRlin}\\
    S_{cor}&\triangleq\mathcal{G}\,\text{div}(\omega\kappa^{(t)}+(\omega^{(c)}+\omega^{(t)})\kappa^{(p)}),\label{DEFofScor}\\
    S_{lin}&\triangleq\mathcal{G}(\partial_t\kappa^{(p)}+(-\Delta)^{\alpha}\kappa+\text{div}(\bar{u}\kappa+\omega\bar{\theta})).\label{DEFofSlin}
\end{align}

Finally, we let the velocity $\tilde{u}$ and $\tilde{\theta}$ of Proposition \ref{convexi} be
\begin{gather*}
    \tilde{u}=\bar{u}+\omega,\quad\tilde{\theta}=\bar{\theta}+\kappa
\end{gather*}
which leads to the Boussinesq-Reynolds system
\begin{gather*}
    \begin{cases}
    \partial_t\tilde{u}+\text{div}(\tilde{u}\otimes \tilde{u})+\nabla \tilde{p}+(-\Delta)^{\alpha} \tilde{u}=\tilde{\theta} e_d+\text{div}\tilde{R}, \quad\quad \\
    \text{div}\,\tilde{u}=0,\\
    \partial_t\tilde{\theta}+\text{div}(\tilde{u}\tilde{\theta})+(-\Delta)^{\alpha} \tilde{\theta}=\text{div}\tilde{S},
    \end{cases}
\end{gather*}
with
\begin{align*}
    \tilde{p}&=\bar{p}-f^2g_l^2(\nu t)\rho+\nu^{-1}\Delta^{-1}\text{div}\text{div}(\partial_t\bar{R}h_l(\nu t)),\\
    \tilde{R}&=R_{far}+R_{osc,t}+R_{osc,x}+R_{cor}+R_{lin},\\
    \tilde{S}&=S_{far}+S_{osc,t}+S_{osc,x}+S_{cor}+S_{lin}.
\end{align*}

Here we briefly check the equation with temperature in the Boussinesq-Reynolds system. By (\ref{666}) and the definition of $\tilde{S}$, we have
\begin{align*}
    \partial_t\tilde{\theta}+\text{div}(\tilde{u}\tilde{\theta})+(-\Delta)^{\alpha} \tilde{\theta}=&\partial_t\bar{\theta}+\text{div}(\bar{u}\bar{\theta})+(-\Delta)^{\alpha} \bar{\theta}+\partial_t\kappa+\text{div}(\bar{u}\kappa+\omega\bar{\theta}+\omega\kappa)+(-\Delta)^{\alpha}\kappa\\
    =&\text{div}\bar{S}+\partial_t\kappa^{(t)}+\text{div}(\omega^{(p)}\kappa^{(p)})+\text{div}(\bar{u}\kappa+\omega\bar{\theta})+\partial_t\mathbb{P}_{\neq 0}\kappa^{(p)}+(-\Delta)^{\alpha}\kappa\\&+\text{div}\left[\omega\kappa^{(t)}-\omega\fint_{\mathbb{T}^d}k^{(p)}+(\omega^{(c)}+\omega^{(t)})\kappa^{(p)}\right]\\
    =&\text{div}\tilde{S},
\end{align*}
where $\text{div}(\omega\fint_{\mathbb{T}^d}k^{(p)})=0$ is due to the fact that $\omega$ is divergence-free.

\section{Proof of Proposition \ref{convexi}}\label{section3}
In this section, we complete the proof of Proposition \ref{convexi}. 
Since the claim that $(\tilde{u},\tilde{R},\tilde{\theta},\tilde{S})$ of (\ref{e:Boussinesq-Reynold}) are well-prepared for the set $\bar{I}$ and the same length scale $\bar{\tau}$ and $\text{supp}_t(\omega,\kappa)\subset \bar{I}$ follows easily from the definition, we omit the proof. We only need to choose suitable parameters and prove that the new perturbations and the Reynolds stress constructed in Section \ref{section of convex} satisfy estimates of Proposition \ref{convexi}. We denote $\bar{C}$ are those constants independent of the parameters $(\nu,\sigma,l,\mu)$ but depending on $(\bar{u},\bar{\theta},\bar{R},\bar{S})$. We recall the exponents $p$ and $q$ are the ones given by Proposition \ref{convexi}.

\subsection{Choice of parameters}

We fix two parameters $\beta\in(0,\frac{1}{2})$ and $r\in(1,2)$ as follows:

\noindent
\textbf{Case A}: $1<\alpha<\frac{d+1}{2}$

\noindent
Firstly, the first parameter $\beta>0$ satisfies
\begin{gather}
   (\frac{1}{2}-\frac{1}{p})\frac{2\alpha}{2\alpha-2}(d-1)+\frac{d-1}{2}-(\frac{1}{2}-\frac{1}{p})(6+\frac{5}{\alpha-1})\beta<-\beta,\label{1001}\\
   -(\frac{1}{2}-\frac{2\alpha-1}{2\alpha})(6+\frac{5}{\alpha-1})\beta-\frac{d-1}{q}<-\beta\label{1010},\\
   \frac{2\alpha}{2\alpha-2}(d-1)-(6+\frac{5}{\alpha-1})\beta>2\beta,\label{1005}\\
   \frac{d-1}{2\alpha-2}-1-\frac{5\beta}{2\alpha-2}>2\beta.\label{1004}
\end{gather}

\noindent
Then, fixing $\beta$ as above, we choose $r\in(1,2)$ such that
\begin{gather}
d-\frac{d}{r}<\beta.\label{chooseofr}
\end{gather} 

Since $p<\frac{2\alpha}{2\alpha-1}$ and $1<\alpha<\frac{d+1}{2}$, we easily deduce
\begin{gather*}
    (\frac{1}{2}-\frac{1}{p})\frac{2\alpha}{2\alpha-2}(d-1)+\frac{d-1}{2}<0,\quad
    \frac{2\alpha}{2\alpha-2}(d-1)>0,\quad
    \frac{d-1}{2\alpha-2}-1>0.
\end{gather*}
Hence, we can choose $\beta>0$ sufficiently small to satisfy the above.

In case A, We choose the parameter $l,\nu,\sigma,\mu$ as follows,
    \begin{align*}
   \nu&\triangleq\lambda^{\beta},\\
   l&\triangleq\lambda^{\frac{2\alpha}{2\alpha-2}(d-1)-(6+\frac{5}{\alpha-1})\beta},\\
    \sigma&\triangleq\lambda^{\frac{d-1}{2\alpha-2}-1-\frac{5\beta}{2\alpha-2}}, 
   \\ \mu&\triangleq\lambda.
    \end{align*}

\noindent
\textbf{Case B}: $\alpha=1$

\noindent
We choose $\beta>0$ and $r$ which satisfy
\begin{gather*}
     10\beta<d-1\quad\text{and}\quad d-\frac{d}{r}<\beta.
\end{gather*}
In case B, We choose the parameter $l,\nu,\sigma,\mu$ as follows.
    \begin{align*}
        \nu&\triangleq\lambda^{\beta},\\
        l&\triangleq\lambda^{\text{max}\{2,\frac{d-1+2\beta}{2(\frac{1}{p}-\frac{1}{2})}\}},\\
        \sigma&\triangleq l^{\frac{1}{2}}\mu^{-1+2\beta}=\lambda^{\frac{1}{2}\text{max}\{2,\frac{d-1+2\beta}{2(\frac{1}{p}-\frac{1}{2})}\}-1+2\beta},\\
        \mu&\triangleq\lambda.
    \end{align*}

In both cases, we choose $\lambda>0$ is large enough so that $\nu,\sigma\in\mathbb{N}$ and the value of $\lambda$ will be fixed in the end. The parameter $\beta$ we choose above immediately implies $l,\nu,\sigma,\mu\gg 1$.

\begin{lemm}\label{lemma of parameters}
    Let $l,\nu,\sigma,\mu$ be chosen above, for any $1\leq\alpha<\frac{d+1}{2}$, there holds
    \begin{align}
        l^{\frac{1}{2}-\frac{1}{p}}\mu^{\frac{d-1}{2}}&\leq\lambda^{-\beta},\label{101}\\
        \sigma^{-1}(\nu l)l^{-\frac{1}{2}}\mu^{-1+\frac{d-1}{2}-\frac{d-1}{r}}&\leq\lambda^{-\beta},\label{102}\\
        l^{-\frac{1}{2}}(\sigma\mu)^{2\alpha-1}\mu^{\frac{d-1}{2}-\frac{d-1}{r }}&\leq\lambda^{-\beta}.\label{103}\\
        \sigma^{-1}\mu^{d-1-\frac{d-1}{r}}&\leq\lambda^{-\beta},\label{104}\\
        l^{-\frac{1}{2}}&\leq\lambda^{-\beta}.\label{105}
    \end{align} 
    Moreover, if $1<\alpha<\frac{d+1}{2}$, we have
    \begin{align*}
        l^{\frac{1}{2}-\frac{2\alpha-1}{2\alpha}}\mu^{\frac{d-1}{2}-\frac{d-1}{q}}\leq\lambda^{-\beta}.\label{110}
    \end{align*}
\end{lemm}
\begin{proof}
    We only prove the inequalities in Case A (i.e. $1<\alpha<\frac{d+1}{2}$) and the proof in Case B is similar.
    For the first inequality (\ref{101}), it suffices to show that
    \begin{gather*}
        \left(\frac{2\alpha}{2\alpha-2}(d-1)-(6+\frac{5}{\alpha-1})\beta\right)(\frac{1}{2}-\frac{1}{p})+\frac{d-1}{2}\leq-\beta,
    \end{gather*}
    which is exactly identical to (\ref{1001}).

    The second inequality (\ref{102}) is equivalent to
    \begin{gather*}
        -(\frac{d-1}{2\alpha-2}-1-\frac{5\beta}{2\alpha-2})+\beta+\frac{1}{2}\left(\frac{2\alpha}{2\alpha-2}(d-1)-(6+\frac{5}{\alpha-1})\beta\right)-1+\frac{d-1}{2}-\frac{d-1}{r}\leq-\beta.
    \end{gather*}
    Thus, direct computation
    \begin{gather*}
         -(\frac{d-1}{2\alpha-2}-1-\frac{5\beta}{2\alpha-2})+\beta+\frac{1}{2}\left(\frac{2\alpha}{2\alpha-2}(d-1)-(6+\frac{5}{\alpha-1})\beta\right)-1-\frac{d-1}{2}=-2\beta
    \end{gather*}
    and (\ref{chooseofr}) immediately imply (\ref{102}).

    Similarly, the third inequality (\ref{103}) is equivalent to
    \begin{gather*}
        -\frac{1}{2}\left(\frac{2\alpha}{2\alpha-2}(d-1)-(6+\frac{5}{\alpha-1})\beta\right)+\left(\frac{d-1}{2\alpha-2}-\frac{5\beta}{2\alpha-2}\right)(2\alpha-1)+\frac{d-1}{2}-\frac{d-1}{r}\leq-\beta.
    \end{gather*}
    Thus, (\ref{chooseofr}) and the fact that
    \begin{gather*}
        -\frac{1}{2}\left(\frac{2\alpha}{2\alpha-2}(d-1)-(6+\frac{5}{\alpha-1})\beta\right)+\left(\frac{d-1}{2\alpha-2}-\frac{5\beta}{2\alpha-2}\right)(2\alpha-1)-\frac{d-1}{2}=-2\beta
    \end{gather*}
    give (\ref{103}).
    
      Moreover, (\ref{1005}),(\ref{1004}) and (\ref{chooseofr}) imply directly (\ref{104}) and (\ref{105}).

     In the case $1<\alpha<\frac{d+1}{2}$, by (\ref{1010}), we have
     \begin{align*}
       &\quad(\frac{1}{2}-\frac{2\alpha-1}{2\alpha}) \left( \frac{2\alpha}{2\alpha-2}(d-1)-(6+\frac{5}{\alpha-1})\beta\right)+\frac{d-1}{2}-\frac{d-1}{q}\\&=-(\frac{1}{2}-\frac{2\alpha-1}{2\alpha})(6+\frac{5}{\alpha-1})\beta-\frac{d-1}{q}\\&<-\beta,
     \end{align*}
     which immediately implies (\ref{110}).
\end{proof}

\subsection{Estimates on perturbation}
We first show the estimate on amplitude functions $a_k$ and $b_k$, their definitions and Proposition \ref{prop of gl} immediately give the following results.

\begin{lemm}\label{lemma of akbk}
For any $r\in[1,\infty]$ and $m,n\in\mathbb{N}$, 
    \begin{gather*}
        \|(\partial_t^n\nabla^m a_k,\partial_t^n\nabla^mb_k)\|_{L^r_TL^\infty_x}\leq \bar{C}(\nu l)^nl^{\frac{1}{2}-\frac{1}{r}}.
    \end{gather*}
    Moreover, we have
    \begin{gather*}
        \|(a_k(t),b_k(t))\|_{L^2(\mathbb{T}^d)}\lesssim g_l(\nu t)(\int_{\mathbb{T}^d}\rho(t,x)dx)^{\frac{1}{2}}.
    \end{gather*}
\end{lemm}

Using Proposition \ref{mikado} and lemma \ref{lemma of akbk}, we obtain
the following estimate on the principal part of the perturbation.
\begin{prop}\label{prop of wp}
    Let $\delta>0$. If $\lambda$ is sufficiently large, then we have
    \begin{gather*}
        \|(\omega^{(p)},\mathbb{P}_{\neq 0}\kappa^{(p)})\|_{L^p_TL^\infty_x}\leq \delta,\\
        \|(\omega^{(p)},\mathbb{P}_{\neq 0}\kappa^{(p)})\|_{L^2_TL^2_x}\lesssim\|(\bar{R},\bar{S})\|_{L^1_{t,x}}^{\frac{1}{2}}.
    \end{gather*}
    Moreover, if $1<\alpha<\frac{d+1}{2}$, then we have
    \begin{gather*}
        \|(\omega^{(p)},\mathbb{P}_{\neq 0}\kappa^{(p)})\|_{L^\frac{2\alpha}{2\alpha-1}_TL^q_x}\leq \delta,
    \end{gather*}
\end{prop}

\begin{proof}
    We first focus on the estimate of $\kappa^{(p)}$. By the definition of $\kappa^{(p)}$, Proposition \ref{mikado} and Lemma \ref{lemma of akbk}, we have
    \begin{align*}
        \|\mathbb{P}_{\neq 0}\kappa^{(p)}\|_{L^p_TL^\infty_x}&\leq\|\kappa^{(p)}\|_{L^p_TL^\infty_x}+\|\mathbb{P}_{= 0}\kappa^{(p)}\|_{L^p_TL^\infty_x}\\
        &\lesssim \sum_{k\in\Lambda}\|b_k\|_{L^p_TL^\infty_x}\|\Psi_k^\mu\|_{L^\infty_x}+\|b_k\|_{L^p_TL^\infty_x}\|\Psi_k^\mu\|_{L^1_x}\\
        &\leq \bar{C}\mu^{\frac{d-1}{2}}l^{\frac{1}{2}-\frac{1}{p}}+\bar{C}\mu^{-\frac{d-1}{2}}l^{\frac{1}{2}-\frac{1}{p}},
    \end{align*}
    which by Lemma \ref{lemma of parameters} implies that
    \begin{gather}
         \|\mathbb{P}_{\neq 0}\kappa^{(p)}\|_{L^p_TL^\infty_x}\leq\bar{C}\lambda^{-\beta}.\label{shiziA}
    \end{gather}
    
    Similarly, if $1<\alpha<\frac{d+1}{2}$, we have 
\begin{align}
        \|\mathbb{P}_{\neq 0}\kappa^{(p)}\|_{L^\frac{2\alpha}{2\alpha-1}_TL^q_x}&\leq\|\kappa^{(p)}\|_{L^\frac{2\alpha}{2\alpha-1}_TL^q_x}+\|\mathbb{P}_{= 0}\kappa^{(p)}\|_{L^\frac{2\alpha}{2\alpha-1}_TL^q_x}\nonumber\\
        &\lesssim\sum_{k\in\Lambda}\|b_k\|_{L^\frac{2\alpha}{2\alpha-1}_TL^\infty_x}\|\Psi_k^\mu\|_{L^q_x}+\|b_k\|_{L^\frac{2\alpha}{2\alpha-1}_TL^\infty_x}\|\Psi_k^\mu\|_{L^1_x}\nonumber\\
        &\leq\bar{C}\mu^{\frac{d-1}{2}-\frac{d-1}{q}}l^{\frac{1}{2}-\frac{2\alpha-1}{2\alpha}}+\bar{C}\mu^{-\frac{d-1}{2}} l^{\frac{1}{2}-\frac{2\alpha-1}{2\alpha}},\nonumber\\
        &\leq\bar{C}\lambda^{-\beta}.\label{shiziC}
    \end{align}
    
    In the view of Proposition \ref{mikado}, Lemma \ref{lemma of akbk} and Lemma \ref{improved holder}, we obtain
    \begin{align}
        \|\kappa^{(p)}\|_{L^2({\mathbb{T}^d})}&\lesssim\sum_{k\in\Lambda}\|b_k(t)\|_{L^2({\mathbb{T}^d})}\|\Psi_k^\mu\|_{L^2({\mathbb{T}^d})}+\sigma^{-\frac{1}{2}}\bar{C}\nonumber\\
        &\lesssim g_l(\nu t)(\int_{\mathbb{T}^d}\rho(t,x)dx)^{\frac{1}{2}}+\sigma^{-\frac{1}{2}}\bar{C}.\label{shizi11}
    \end{align}
    Applying Lemma \ref{improved holder} once again, we infer that
    \begin{align}
        \int_0^Tg_l^2(\nu t)\int_{\mathbb{T}^d}\rho(t,x)dxdt\lesssim\|(\bar{R},\bar{S})\|_{L^1_{t,x}}+\bar{C}\nu^{-1},\label{shizi22}
    \end{align}
    where we use the fact that $\int g_l^2=1$ and thanks to the definition of $\rho$.

    Thus, combining (\ref{shizi11}) and (\ref{shizi22}), we have
    \begin{align*}
        \|\kappa^{(p)}\|_{L^2_TL^2_x}\lesssim\|(\bar{R},\bar{S})\|_{L^1_{t,x}}^{\frac{1}{2}}+\bar{C}(\nu^{-\frac{1}{2}}+\sigma^{-\frac{1}{2}}).
    \end{align*}
    Moreover, since
    \begin{align*}
        \|\mathbb{P}_{= 0}\kappa^{(p)}\|_{L^2_TL^2_x}&\lesssim\|\kappa^{(p)}\|_{L^2_TL^1_x}\\&\lesssim\sum_{k\in\Lambda}\|b_k\|_{L^2_TL^\infty_x}\|\Psi_k^\mu\|_{L^1_x}\\
        &\leq \bar{C}\mu^{-\frac{d-1}{2}},
    \end{align*}
    we obtain
    \begin{gather}
        \|\mathbb{P}_{\neq 0}\kappa^{(p)}\|_{L^2_TL^2_x}\lesssim\|(\bar{R},\bar{S})\|_{L^1_{t,x}}^{\frac{1}{2}}+\bar{C}(\nu^{-\frac{1}{2}}+\sigma^{-\frac{1}{2}}+\mu^{-\frac{d-1}{2}}).\label{shiziB}
    \end{gather}
    
    Finally, if $\lambda$ is sufficiently large, (\ref{shiziA}), (\ref{shiziC}) and (\ref{shiziB}) give the desired results. The estimation procedure for $\omega^{(p)}$ is analogous to that employed for $\kappa^{(p)}$, so we omit the proof.
    
\end{proof}

Compared to the principal perturbation $\omega^{(p)}$ and $\kappa^{(p)}$, the divergence-free corrector $\omega^{(c)}$ and the temporal corrector $(\omega^{(t)},\kappa^{(t)})$ are much smaller.

\begin{prop}\label{prop of Wc}
    Let $\delta>0$. If $\lambda$ is sufficiently large, then we have
    \begin{gather*}
        \|\omega^{(c)}\|_{{L^p_TL^\infty_x}}\leq \delta,\\
        \|\omega^{(c)}\|_{L^2_TL^2_x}\leq\|(\bar{R},\bar{S})\|_{L^1_{t,x}}^{\frac{1}{2}},
    \end{gather*}
    and
    \begin{gather*}
        \|(\omega^{(t)},\kappa^{(t)})\|_{L^\infty_{t,x}} \leq \delta.
    \end{gather*}
    Moreover, if $1<\alpha<\frac{d+1}{2}$, we have
    \begin{gather*}
        \|\omega^{(c)}\|_{{L^\frac{2\alpha}{2\alpha-1}_TL^q_x}}\leq \delta.
    \end{gather*}
\end{prop}
\begin{proof}
    By the definition of $\omega^{(c)}$, Proposition \ref{mikado} and Lemma \ref{lemma of akbk}, we have 
    \begin{align}
        \|\omega^{(c)}\|_{L^p_TL^\infty_x}&\lesssim\sigma^{-1}\sum_{k\in\Lambda} \|\nabla a_k\|_{L^p_TL^\infty_x}\|\Omega^\mu_k(\sigma \cdot)\|_{L^\infty_x}\nonumber\\&\leq\bar{C}\sigma^{-1}l^{\frac{1}{2}-\frac{1}{p}}\mu^{-1+\frac{d-1}{2}}\nonumber\\&\leq\bar{C}\sigma^{-1}\mu^{-1},\label{A1} 
    \end{align}
    and
    \begin{align}
        \|\omega^{(c)}\|_{L^2_{t,x}}&\lesssim\sigma^{-1}\sum_{k\in\Lambda} \|\nabla a_k\|_{L^2_TL^\infty_x}\|\Omega^\mu_k(\sigma \cdot)\|_{L^2_x}\nonumber\\&\leq\bar{C}\sigma^{-1}\mu^{-1}.\label{A2} 
    \end{align}
    By the definition of $(\omega^{(t)},\kappa^{(t)})$, Proposition \ref{prop of gl} and Proposition \ref{mikado}, we deduce
    \begin{align}
        \|(\omega^{(t)},\kappa^{(t)})\|_{L^\infty_{t,x}}\leq\bar{C}\nu^{-1}\|h\|_{L^\infty_{T}}\leq\bar{C}\nu^{-1}.\label{A3}
    \end{align}

    Similarly, if $1<\alpha<\frac{d+1}{2}$, we have
    \begin{align}
        \|\omega^{(c)}\|_{{L^\frac{2\alpha}{2\alpha-1}_TL^q_x}}&\lesssim\sigma^{-1}\sum_{k\in\Lambda} \|\nabla a_k\|_{L^\frac{2\alpha}{2\alpha-1}_TL^\infty_x}\|\Omega^\mu_k(\sigma \cdot)\|_{L^q_x}\nonumber\\
        &\leq\bar{C}\sigma^{-1}\mu^{-1+\frac{d-1}{2}-\frac{d-1}{q}}l^{\frac{1}{2}-\frac{2\alpha-1}{2\alpha}}\nonumber\\&\leq\bar{C}\sigma^{-1}\mu^{-1}.\label{A4}
    \end{align}
    
    Thus, if $\lambda$ is sufficiently large, (\ref{A1}), (\ref{A2}), (\ref{A3}) and (\ref{A4}) give the desired results.

\end{proof}

Combining Proposition (\ref{prop of wp}) and Proposition (\ref{prop of Wc}), we easily derive the following proposition.
\begin{prop}\label{propperturbation}
    Let $\delta>0$. If $\lambda$ is sufficiently large, then we have
    \begin{gather*}
        \|(\omega,\kappa)\|_{L^p_TL^\infty_x}\leq\delta,\\
         \|(\omega,\kappa)\|_{L^2_TL^2_x}\lesssim\|(\bar{R},\bar{S})\|_{L^1_{t,x}}^{\frac{1}{2}}.
    \end{gather*}
    Moreover, if $1<\alpha<\frac{d+1}{2}$, we have
    \begin{gather*}
        \|(\omega,\kappa)\|_{{L^\frac{2\alpha}{2\alpha-1}_TL^q_x}}\leq \delta.
    \end{gather*}
\end{prop}

\subsection{Estimate on Reynolds stress error}\label{errorsection}
Having completed the estimates for the perturbation, we now turn our attention to estimating the Reynolds errors $\tilde{R}$ and $\tilde{S}$ in this part.
\begin{prop}\label{propofRosc}
    Let $\delta>0$. If $\lambda$ is sufficiently large, then we have
    \begin{gather*}
        \|(R_{osc,t},S_{osc,t})\|_{L^1_TL^r_x}\leq\delta,\\
        \|(R_{osc,x},S_{osc,x})\|_{L^1_TL^r_x}\leq\delta,\\
        \|(R_{far},S_{far})\|_{L^1_TL^r_x}\leq\delta.
    \end{gather*}
\end{prop}
\begin{proof}
    The first estimate follows easily from the definition of $R_{osc,t}$ and $S_{osc,t}$ and Proposition \ref{prop of gl}, so we omit the proof.

     In the view of (\ref{DEFofSoscx}), Proposition \ref{mikado}, Lemma \ref{lemma of akbk}, Lemma \ref{bilinear B} and Lemma \ref{theo of antidiv}, we deduce
    \begin{align*}
        \|S_{osc,x}\|_{L^1_TL^r_x}&=\|\sum_{k\in\Lambda}\tilde{\mathcal{B}}(e_k\nabla(a_kb_k),\mathbb{P}_{\neq0}\Psi_k^\mu (\sigma x)^2)\|_{L^1_TL^r_x}\\&\lesssim\sum_{k\in\Lambda}\|\nabla(a_kb_k)\|_{L^1_TC^1_x}\|\mathcal{G}\,\mathbb{P}_{\neq0}\Psi_k^\mu (\sigma x)^2\|_{L^r_x}\\&\lesssim\sigma^{-1}\sum_{k\in\Lambda}\|\nabla(a_kb_k)\|_{L^1_TC^1_x}\|\Psi_k^\mu\|_{L^{2r}_x}^2\\&\leq\bar{C}\sigma^{-1}\mu^{d-1-\frac{d-1}{r}}\\&\leq\bar{C}\lambda^{-\beta}.
    \end{align*}
    Since the estimate of $R_{osc,x}$ is similar, for $\lambda$ sufficiently large, there holds
    \begin{gather*}
        \|(R_{osc,x},S_{osc,x})\|_{L^1_TL^r_x}\leq\delta.
    \end{gather*}
    
    By (\ref{DEFofSfar}), Proposition \ref{mikado} and Lemma \ref{lemma of akbk}, we have
    \begin{align*}
        \|S_{far}\|_{L^1_TL^r_x}&\lesssim\sum_{k\neq k'}\|a_k\|_{L^2_TL^\infty_x}\|b_{k'}\|_{L^2_TL^\infty_x}\|\Psi_k^\mu\Psi_{k'}^\mu\|_{L^r_x}\nonumber\\&\leq\bar{C}\mu^{d-1-\frac{d}{r}}\nonumber\\&\leq\bar{C}\lambda^{\beta-1},
    \end{align*}
    where the last inequality follows from (\ref{chooseofr}) and estimate of $R_{far}$ is similar. Taking $\lambda$ sufficiently large, there holds
    \begin{gather*}
        \|(R_{far},S_{far})\|_{L^1_TL^r_x}\leq\delta.
    \end{gather*}
\end{proof}

\begin{prop}
     Let $\delta>0$. If $\lambda$ is sufficiently large, then we have
    \begin{gather*}
        \|(R_{lin},S_{lin})\|_{L^1_TL^r_x}\leq\delta.
    \end{gather*}
\end{prop}
\begin{proof}
We first divide the linear error $S_{lin}$ into three parts:
    \begin{align*}
        \|S_{lin}\|_{L^1_TL^r_x}&\leq\|\mathcal{G}\partial_t\kappa^{(p)}\|_{L^1_TL^r_x}+\|\mathcal{G}(-\Delta)^{\alpha}\kappa\|_{L^1_TL^r_x}+\|\mathcal{G}\,\text{div}(\bar{u}\kappa+\omega\bar{\theta})\|_{L^1_TL^r_x}\\
        &\triangleq L_1+L_2+L_3.
    \end{align*}

\noindent
\textbf{Estimate of} ${L_1}$:

Using (\ref{DEFofKP}), Proposition \ref{mikado}, Lemma \ref{lemma of akbk}, Lemma \ref{theo of antidiv} and $L^r_x$ boundedness of $\mathcal{G}\,\text{div}$, we deduce  
\begin{align}
    L_1&=\|\mathcal{G}\sum_{k\in\Lambda}\partial_tb_k\Psi_k^\mu (\sigma \cdot)\|_{L^1_TL^r_x}\nonumber\\
    &=\sigma^{-1}\left\|\mathcal{G}\left(\sum_{k\in\Lambda}\text{div}(\partial_tb_k\cdot\nabla\Phi_k^\mu (\sigma \cdot))-\nabla\partial_tb_k\cdot\nabla\Phi_k^\mu (\sigma \cdot)\right)\right\|_{L^1_TL^r_x}\nonumber\\&\lesssim \sigma^{-1}\left(\sum_{k\in\Lambda}\|\partial_tb_k\|_{L^1_TL^\infty_x}\|\nabla\Phi_k^\mu (\sigma \cdot)\|_{L^r_x}+\|\nabla\partial_tb_k\|_{L^1_TL^\infty_x}\|\nabla\Phi_k^\mu (\sigma \cdot)\|_{L^r_x}\right)\nonumber\\
    &\leq\bar{C}\sigma^{-1}(\nu l)l^{-\frac{1}{2}}\mu^{-1+\frac{d-1}{2}-\frac{d-1}{r}}\nonumber\\
    &\leq\bar{C}\lambda^{-\beta},\label{A}
\end{align}
where the last inequality follows from Lemma \ref{lemma of parameters}.

\noindent
\textbf{Estimate of} ${L_2}$:

By $L^r_x$ boundedness of $\mathcal{G}\,\nabla$, we deduce
\begin{align*}
    L_2&\lesssim\||\nabla|^{2\alpha-1}\kappa^{(p)}\|_{L^1_TL^r_x}+\||\nabla|^{2\alpha-1}\kappa^{(t)}\|_{L^1_TL^r_x}.
\end{align*}
Thanks to the interpolation estimate, Proposition \ref{mikado}, Lemma \ref{lemma of parameters}, Lemma \ref{lemma of akbk}, we have
\begin{align}
     L_2\lesssim\bar{C}l^{-\frac{1}{2}}(\sigma\mu)^{2\alpha-1}\mu^{\frac{d-1}{2}-\frac{d-1}{r }}+\bar{C}l^{-\frac{1}{2}}\leq\bar{C}\lambda^{-\beta}.\label{B}
\end{align}

\noindent
\textbf{Estimate of} ${L_3}$:

By $L^r_x$ boundedness of $\mathcal{G}\,\text{div}$, we have
\begin{align}
    L_3\lesssim\|\bar{u}\kappa+\omega\bar{\theta}\|_{L^1_TL^r_x}\lesssim\bar{C}\|(\omega,\kappa)\|_{L^p_TL^\infty_x}.\label{C}
\end{align}

Finally, combining (\ref{A}), (\ref{B}), (\ref{C}) and Proposition \ref{propperturbation} implies the desired result of $S_{lin}$ if $\lambda$ is sufficiently large. Since the estimate of $R_{lin}$ follows a similar logic, we omit the corresponding proof.
\end{proof}

\begin{prop}\label{propofRcor}
     Let $\delta>0$. If $\lambda$ is sufficiently large, then we have
    \begin{gather*}
        \|(R_{cor},S_{cor})\|_{L^1_TL^r_x}\leq\delta.
    \end{gather*}
\end{prop}
\begin{proof}
    By (\ref{DEFofScor}), $L^r_x$ boundedness of $\mathcal{G}\,\text{div}$ and H\"{o}lder's inequality, we deduce
    \begin{align*}
        \|S_{cor}\|_{L^1_TL^r_x}\lesssim(\|\omega\|_{L^2_{t,x}}+\|\kappa^{(p)}\|_{L^2_{t,x}})(\|\omega^{(c)}\|_{L^2_TL^m_x}+\|\omega^{(t)}\|_{L^\infty_{t,x}}+\|\kappa^{(t)}\|_{L^\infty_{t,x}}),
    \end{align*}
    where $q$ satisfies
    \begin{align*}
        \frac{1}{r}=\frac{1}{2}+\frac{1}{m}.
    \end{align*}
    Therefore, due to Proposition \ref{prop of wp}, Proposition \ref{prop of Wc} and Proposition \ref{propperturbation}, we have
    \begin{align}
        \|S_{cor}\|_{L^1_TL^r_x}\leq\bar{C}(\sigma^{-1}\mu^{-1+d-1-\frac{d-1}{r}}+\nu^{-1})\leq\bar{C}\lambda^{-\beta},\label{Scor}
    \end{align}
    Thanks to the estimate of $\omega^{(c)}$:
    \begin{align*}
        \|\omega^{(c)}\|_{L^2_TL^m_x}\lesssim\sigma^{-1}\sum_{k\in\Lambda} \|\nabla a_k\|_{L^2_TL^\infty_x}\|\Omega^\mu_k(\sigma \cdot)\|_{L^m_x}\leq\bar{C}\sigma^{-1}\mu^{-1+d-1-\frac{d-1}{r}}.
    \end{align*}
    Hence, the desired result of $S_{cor}$ follows from (\ref{Scor}) if $\lambda$ is sufficiently large. As usual, the estimate of $R_{cor}$ is similar, so we omit the proof.
\end{proof}

Combining Proposition \ref{propofRosc}-\ref{propofRcor}, we can easily get the following result.
\begin{prop}
    Let $\delta>0$. If $\lambda$ is sufficiently large, then we have
    \begin{gather*}
        \|(\tilde{R},\tilde{S})\|_{L^1_{t,x}}\leq\delta.
    \end{gather*}
\end{prop}

\section{Proof of main results}\label{section4}
In this section, we give a detailed proof of Proposition \ref{mainProp}, Theorem \ref{maintheo}, Theorem \ref{maintheo2} and Theorem \ref{正面}.

\begin{proof}[\textbf{Proof of Proposition \ref{mainProp}}]
    Let $v$ and $\rho$ be the smooth functions given in Proposition \ref{mainProp}. We define $(u_0,\theta_0,R_0,S_0)$ as follows,
    \begin{gather*}
        u_0=v,\\
        \theta_0=\rho,\\
        R_0=\mathcal{R}\left(\partial_tu_0+(-\Delta)^{\alpha}u_0+\text{div}(u_0\otimes u_0)-\theta_0 e_d\right),\\
        S_0=\mathcal{G}\left(\partial_t\theta_0+(-\Delta)^{\alpha}\theta_0+\text{div}(u_0\theta_0)\right),
    \end{gather*}
    where $\mathcal{R}$ and $\mathcal{G}$ are inverse divergence operator on $\mathbb{T}^d$ defined in Lemma \ref{def of antidiv}.

    Since $(u_0,\theta_0)$ are mean-free, $(u_0,\theta_0,R_0,S_0)$ solves (\ref{e:Boussinesq-Reynold}). Combining Proposition \ref{prop of gluing} and Proposition \ref{convexi}, we can construct a sequence of well-preparedness weak solutions $\{(u_n,\theta_n,R_n,S_n)\}$ the set $I_n$ and length scale $\tau_n$ as follows. For any $n\in\mathbb{N}^*$, $\{(u_n,\theta_n,R_n,S_n)\}$ satisfy
    \begin{gather*}
        \delta_n\triangleq2^{-n}\,\text{min}\{\|(R_{n-1},S_{n-1})\|_{L^1_TL^r_x},\epsilon\},\\
        \|(R_{n},S_{n})\|_{L^1_TL^r_x}\leq \delta_{n},\\
        \|(u_n-u_{n-1},\theta_n-\theta_{n-1})\|_{L^2_{t,x}}\leq M\|(R_{n-1},S_{n-1})\|_{L^1_TL^r_x}\leq M\delta_{n-1},\\
        \|(u_n-u_{n-1},\theta_n-\theta_{n-1})\|_{L^p_TL^{\infty}_x}\leq\delta_n,\\
        \left(u_n(0,x),\theta_n(0,x)\right)= \left(u(0,x),\theta(0,x)\right).
    \end{gather*}
    Moreover, if $1<\alpha<\frac{d+1}{2}$, we have
     \begin{gather*}
         \|(u_n-u_{n-1},\theta_n-\theta_{n-1})\|_{L^{\frac{2\alpha}{2\alpha-1}}_TL^{q}_x}\leq\delta_n,\quad\forall n\in\mathbb{N}^*.
     \end{gather*}
    Therefore, there exist $(u,\theta)\in L^p_TL^{\infty}_x\cap L^2_{t,x}$ such that
    \begin{gather*}
        (u_n,\theta_n)\rightarrow(u,\theta)\quad\text{in\quad} L^p_TL^{\infty}_x\cap L^2_{t,x},\\
        \|(u,\theta)-(v,\rho)\|_{L^p_TL^{\infty}_x\cap L^2_{t,x}}\leq\epsilon,
    \end{gather*}
    and $(u,\theta)$ is weak solution of (\ref{e:boussinesq equation}).
    If $1<\alpha<\frac{d+1}{2}$, $(u,\theta)$ will additionally belong to ${L^{\frac{2\alpha}{2\alpha-1}}_TL^{q}_x}$ and satisfy
    \begin{gather*}
        \|(u,\theta)-(v,\rho)\|_{L^{\frac{2\alpha}{2\alpha-1}}_TL^{q}_x}\leq\epsilon.
    \end{gather*}

    For the structure of the intervals of regularity of $(u,\theta)$, it is similar to \cite[Theorem 1.11]{serrin准则luo} and we give a brief proof. We denote $I_n$ and $\tau_n$ the set and length scale of the well-preparedness of $(u_n,\theta_n,R_n,S_n)$. Note that for any $n\in\mathbb{N}$
    \begin{gather*}
        \text{supp}_t(u_{n+1}-u_n,\theta_{n+1}-\theta_n)\subset I_n,\\
        \tau_{n+1}<\frac{\tau_n}{2}\quad\text{and}\quad I_n\subset I_{n+1}.
    \end{gather*}

    Let 
    \begin{gather*}
        I\triangleq\bigcup_{n\geq0}I_n^c\backslash\{0,1\},
    \end{gather*}
    where the complement is considered with respect to  $[0,T]$. Due to the fact that each $I_n$ is a finite union of closed intervals, $I$ can be expressed as a union of countably many open intervals:
    \begin{gather*}
        I=\bigcup_{i\geq0}(a_i,b_i).
    \end{gather*}
    
    Since for any $m\geq n$, 
    \begin{gather*}
        (I_n)^c\subset (I_m)^c\subset \left(\text{supp}_t(u_{m+1}-u_m,\theta_{m+1}-\theta_m)\right)^c,
    \end{gather*}
     we have $(u,\theta)=(u_{n},\theta_{n})$ on $I_n^c$ for each $n\in\mathbb{N}$. Hence we deduce that $(u,\theta)\in C^\infty(I\times\mathbb{T}^d)$. 
     
     As each $I_n$ is a union of at most $\tau_n^{-\epsilon}$ many closed intervals of length $5\tau_n$, we have
     \begin{gather*}
         d_{\mathcal{H}}([0,T]\backslash I)=d_{\mathcal{H}}(\bigcap_{n\geq0}I_n)=d_{\mathcal{H}}(\varlimsup_{n\rightarrow\infty}I_n)\leq\epsilon,
     \end{gather*}
     which gives the desired result for the Hausdorff dimension bound.
\end{proof}

With the aid of Proposition \ref{mainProp}, non-uniqueness of weak solutions for the Boussinesq system (\ref{e:boussinesq equation}) can be derived.
\begin{proof}[\textbf{Proof of Theorem\ref{maintheo}}]
    Let $d\geq2$ be the dimension, $1\leq\alpha<\frac{d+1}{2}$, $1\leq p<\frac{2\alpha}{2\alpha-1}$.
    Without loss of generality, we suppose that $(u,\theta)\in L^p_TL^\infty_x$ is a smooth weak solution of (\ref{e:boussinesq equation}).
    Let $f$ be a smooth function on $[0,T]\times\mathbb{T}^d$ with
    \begin{gather*}
        \|f\|_{L^p({[T/2,T]};L^2_x)}=1.
    \end{gather*}
    For any $m\in\mathbb{N}$, We define 
    \begin{align*}
        u_m=u\,,\,
        \theta_m=\theta+m(1-\chi)f,
    \end{align*}
    where $\chi$ is smooth function on $[0,T]$ satisfying
    \begin{gather*}
        \chi(t)=\begin{cases}
            1,\quad\text{if}\,\,0\leq t\leq\frac{T}{4},\\0,\quad\text{if}\,\,\frac{T}{2}\leq t\leq T.
        \end{cases}
    \end{gather*}
    Then by Proposition \ref{mainProp}, there exist weak solutions $\{(\tilde{u}_m,\tilde{\theta}_m)\}$ of (\ref{e:boussinesq equation}) satisfying
    \begin{align*}
        \tilde{u}_m(0)=u_m(0)=u(0)\quad,\quad\|\tilde{u}_m-u_m\|_{L^p_TL^\infty_x}\leq\frac{1}{4},\\
        \tilde{\theta}_m(0)=\theta_m(0)=\theta(0)\quad,\quad\|\tilde{\theta}_m-\theta_m\|_{L^p_TL^\infty_x}\leq\frac{1}{4}.
    \end{align*}
    Moreover, we have for any $m\in\mathbb{N}$
    \begin{align*}
        \|\tilde{\theta}_m\|_{L^p_TL^2_x}&\geq \|\theta_m\|_{L^p_TL^2_x}-\|\tilde{\theta}_m-\theta_m\|_{L^p_TL^\infty_x}\\
        &\geq m\|f\|_{L^p({[T/2,T]};L^2_x)}-\|\theta\|_{L^p_TL^\infty_x}-\|\tilde{\theta}_m-\theta_m\|_{L^p_TL^\infty_x}\\
        &\geq m-\|\theta\|_{L^p_TL^\infty_x}-\frac{1}{4},
    \end{align*}
    and for any $m\neq m'\in\mathbb{N}$
    \begin{align*}
        \|\tilde{\theta}_m-\tilde{\theta}_{m'}\|_{L^p_TL^2_x}&\geq\|\theta_m-\theta_{m'}\|_{L^p_TL^2_x}-\|\theta_m-\tilde{\theta}_{m}\|_{L^p_TL^2_x}-\|\theta_{m'}-\tilde{\theta}_{m'}\|_{L^p_TL^2_x}\\
        &\geq|m-m'|\|f\|_{L^p({[T/2,T]};L^2_x)}-\|\theta_m-\tilde{\theta}_{m}\|_{L^p_TL^\infty_x}-\|\theta_{m'}-\tilde{\theta}_{m'}\|_{L^p_TL^\infty_x}\\
        &\geq|m-m'|-\frac{1}{2}.
    \end{align*}
    
    Thus $(\tilde{u}_m,\tilde{\theta}_m)$ is not a generalized Leray-Hopf weak solution when $m$ is large enough and $(\tilde{u}_m,\tilde{\theta}_m)\neq(\tilde{u}_{m'},\tilde{\theta}_{m'})$ if $m\neq m'$. Therefore, there exist infinitely many non-generalized-Leray-Hopf weak solutions in $L^p_TL^\infty_x(\mathbb{T}^d)$ of (\ref{e:boussinesq equation}) with same initial data.
    \end{proof}

The proof of Theorem \ref{maintheo2} is analogous, except that we replace $L^{p}_TL^\infty_x$ with $L^{\frac{2\alpha}{2\alpha-1}}_TL^q_x$ for $1<\alpha<\frac{d+1}{2}$. Next, before proving Theorem \ref{正面}, we show two Lemmas \ref{lemmaofzheng1}-\ref{Lemmaofzheng2} which easily imply Theorem \ref{正面}.
\begin{lemm}\label{lemmaofzheng1}
    If $(u,\theta)\in L^{\frac{2\alpha}{2\alpha-1}}_TL^\infty_x$ is weak solution of (\ref{e:boussinesq equation}), then $(u,\theta)$ is generalized Leray-Hopf weak solution.
\end{lemm}
\begin{proof}
Using standard Galerkin method, we know there exists a weak solution $(\bar{u},\bar{\theta})$ satisfying
    \begin{equation*}
        \begin{cases}
\partial_t\bar{u}+ u\cdot\nabla \bar{u}+\nabla p+(-\Delta)^{\alpha} \bar{u}=\bar{\theta} e_d, \quad\quad \\
\text{div}\,\bar{u}=0,\\
\partial_t\bar{\theta}+u\cdot\nabla\bar{\theta}+(-\Delta)^{\alpha} \bar{\theta}=0, \quad\forall(t,x)\in [0,T]\times\mathbb{T}^d,\\
\bar{u}(0)=u(0),\\
\bar{\theta}(0)=\theta(0),
        \end{cases}
    \end{equation*}
    and the energy inequality
    \begin{gather*}
        \|\bar{u}(t)\|^2_{L_x^2}+2\int_0^t\|(-\Delta)^{\frac{\alpha}{2}} \bar{u}(s)\|^2_{L^2_x}\,ds\leq \|\bar{u}(0))\|^2_{L^2_x}+2\int^t_0\int_{\mathbb{T}^d}\bar{\theta} e_d\cdot \bar{u}\,dxds,\quad\forall t\in[0,T],\\
        \|\bar{\theta}(t)\|^2_{L_x^2}+2\int_0^t\|(-\Delta)^{\frac{\alpha}{2}} \bar{\theta}(s)\|^2_{L^2_x}\,ds\leq \|\bar{\theta}(0)\|^2_{L^2_x},\quad\forall t\in[0,T].
    \end{gather*}

    Thus we only need to show
    \begin{gather*}
        u\equiv\bar{u}\quad\text{and}\quad\theta\equiv\bar{\theta}.
    \end{gather*}
    We denote 
    \begin{gather*}
        \omega:=\bar{u}-u\quad\text{and}\quad\rho:=\bar{\theta}-\theta,
    \end{gather*}
    and $(\omega,\rho)$ satisfy
    \begin{equation*}
        \begin{cases}
\partial_t\omega+ u\cdot\nabla \omega+\nabla p+(-\Delta)^{\alpha} \omega=\rho e_d, \quad\quad \\
\text{div}\,\omega=0,\\
\partial_t\rho+u\cdot\nabla\rho+(-\Delta)^{\alpha} \rho=0, \quad\forall(t,x)\in [0,T]\times\mathbb{T}^d,\\
\omega(0)=0,\\
\rho(0)=0.
        \end{cases}
        \end{equation*}

    By the definition of weak solution, for any $\phi\in C^\infty_0([0,T)\times\mathbb{T}^d;\mathbb{R}^d)$ and $\psi\in C^\infty_0([0,T)\times\mathbb{T}^d;\mathbb{R})$, we have
    \begin{gather}
            -\int^T_0\int_{\mathbb{T}^d}\omega\cdot(\partial_t\phi-\nu(-\Delta)^{\alpha}\phi+u\cdot\nabla\phi)+\rho e_d\cdot\phi\, dxdt=0,\label{hao1}\\
            -\int^T_0\int_{\mathbb{T}^d}\rho\cdot(\partial_t\psi-\kappa(-\Delta)^{\alpha}\psi+u\cdot\nabla\psi)\,dxdt=0.\label{hao2}
        \end{gather}
\end{proof}
By \cite[lemma B.4]{qupeng}, we know for any $F\in C^\infty_0([0,T]\times\mathbb{T}^d;\mathbb{R}^d)$ and $f\in C^\infty_0([0,T]\times\mathbb{T}^d;\mathbb{R})$, there exist $(\Phi,\Psi)\in L^\infty_TL^2_x\cap L^2_T \dot{H}^\alpha$ satisfying
\begin{equation*}
    \begin{cases}
        \partial_t\Phi-(-\Delta)^{\alpha}\Phi+u\cdot\nabla\Phi=F,\\
        \text{div}\,\Phi=0,\\
        \Phi(T)=0,
    \end{cases}
\end{equation*}
and
\begin{equation*}
    \begin{cases}
        \partial_t\Psi-(-\Delta)^{\alpha}\Psi+u\cdot\nabla\Psi=f,\\
        \Psi(T)=0.
    \end{cases}
\end{equation*}

 Therefore, we can use $\phi=\Phi$ and $\psi=\Psi$ respectively in (\ref{hao1}) and (\ref{hao2}) as test functions, resulting in $\omega=0$ and $\rho=0$, which finishes the proof.

\begin{lemm}\label{Lemmaofzheng2}
    Let $d\geq2$, $1\leq\alpha<\frac{d+2}{2}$ and $(u_1,\theta_1)$, $(u_2,\theta_2)$ be two generalized Leray-Hopf weak solutions of (\ref{e:boussinesq equation}) with same initial data. If $(u_1,\theta_1)$ belongs to $L^{\frac{2\alpha}{2\alpha-1}}_TL^\infty_x$, then $(u_1,\theta_1)\equiv(u_2,\theta_2)$.
\end{lemm}
    \begin{proof}
        Let $\omega:= u_1-u_2$ and $\rho:=\theta_1-\theta_2$, we derive that
        \begin{align*}
            \begin{cases}
                \partial_t\omega+(-\Delta)^\alpha\omega+u_1\cdot\nabla u_1-u_2\cdot\nabla u_2+\nabla P=\rho e_d,\\
                \text{div}\,\omega=0,\\
                \omega(0)=0.
            \end{cases}
        \end{align*}
        and
        \begin{align*}
            \begin{cases}
                \partial_t\rho+(-\Delta)^\alpha\rho+u_1\cdot\nabla\theta_1-u_2\cdot\nabla\theta_2=0,\\
                \rho(0)=0.
            \end{cases}
        \end{align*}
        
Using the mollification argument and the fact that $(u_1,\theta_1)\in L^{\frac{2\alpha}{2\alpha-1}}_TL^\infty_x$ and $(u_2,\theta_2)$ is generalized Leray-Hopf weak solution (see \cite[Chapter 12]{NS21century}), we have
        \begin{align}
            \|\omega(t)\|_{L^2_x}^2+2\int^t_0\|\nabla^\alpha\omega(s)\|_{L^2_x}^2ds\leq 2\left|\int^t_0\int_{\mathbb{T}^d}u_1\cdot(\omega\cdot\nabla)\omega\, dxds\right|+2\left|\int^t_0\int_{\mathbb{T}^d}\rho e_d\cdot\omega\,dxds\right|,\label{S1}\\
            \|\rho(t)\|_{L^2_x}^2+2\int^t_0\|\nabla^\alpha\rho(s)\|_{L^2_x}^2ds\leq2\left|\int^t_0\int_{\mathbb{T}^d}\theta_1\cdot(\omega\cdot\nabla)\rho\, dxds\right|\label{S2}.
        \end{align} 
        
        By H\"{o}lder's inequality and interpolation inequality, we obtain
        \begin{align*}
            \|u_1\cdot(\omega\cdot\nabla)\omega\|_{L^1_{t,x}}&\lesssim\|u_1\|_{L^{\frac{2\alpha}{2\alpha-1}}((0,t);L^\infty_x)}\|\omega\|_{L^\infty_tL^2_x}\|\nabla\omega\|_{L^{2\alpha}_tL^2_x}\\
            &\lesssim\|u_1\|_{L^{\frac{2\alpha}{2\alpha-1}}((0,t);L^\infty_x)}\|\omega\|_{L^\infty_tL^2_x}^{2-\frac{1}{\alpha}}\|\nabla^\alpha\omega\|_{L^2_{t,x}}^{\frac{1}{\alpha}},
       \end{align*}
       and 
       \begin{align*}
           \|\theta_1\cdot(\omega\cdot\nabla)\rho\|_{L^1_{t,x}}&\lesssim\|\theta_1\|_{L^{\frac{2\alpha}{2\alpha-1}}((0,t);L^\infty_x)}\|\omega\|_{L^\infty_tL^2_x}\|\nabla\rho\|_{L^{2\alpha}_tL^2_x}\\
           &\lesssim\|\theta_1\|_{L^{\frac{2\alpha}{2\alpha-1}}((0,t);L^\infty_x)}\|\omega\|_{L^\infty_tL^2_x}\|\rho\|_{L^\infty_tL^2_x}^{1-\frac{1}{\alpha}}\|\nabla^\alpha\rho\|_{L^2_{t,x}}^{\frac{1}{\alpha}}.
       \end{align*}
        
        We denote
        \begin{gather*}
            f(t):=\max_{s\in[0,t]}\|\omega(s)\|_{L^2_x}^2+2\int^t_0\|\nabla^\alpha\omega(s)\|_{L^2_x}^2ds+\max_{s\in[0,t]}\|\rho(s)\|_{L^2_x}^2+2\int^t_0\|\nabla^\alpha\rho(s)\|_{L^2_x}^2ds.
        \end{gather*}
        Thus, using Young's inequality and combining (\ref{S1}) and (\ref{S2}), we obtain
        \begin{gather*}
            f(t)\leq C\|(u_1,\theta_1)\|_{L^{\frac{2\alpha}{2\alpha-1}}((0,t);L^\infty_x)}f(t)+C\int^t_0f(s)ds.
        \end{gather*}

        Since $\|(u_1,\theta_1)\|_{L^{\frac{2\alpha}{2\alpha-1}}((0,t);L^\infty_x)}$ tend to 0 when $t$ tend to 0, we obtain $(\omega,\rho)=0$ when $t$ is small by Gronwall's inequality.
        Finally, using continuity argument, we know $(\omega,\rho)=0$ for any $t\in[0,T]$, which finish the proof.

    \end{proof}
    \begin{proof}[\textbf{Proof of Theorem \ref{正面}}]
        Combing Lemma \ref{lemmaofzheng1} and Lemma \ref{Lemmaofzheng2}, we immediately obtain Theorem \ref{正面}.
    \end{proof}
\appendix

\section{Geometric lemma}\label{appendix A}
The geometric lemma from \cite{D1} plays an important role in the convex integration method. Here we recall an extension version of the geometric lemma from \cite{T3}, which is suitable for the Boussinesq equation. We denote that $S_0^{d\times d}$ is the set of symmetric $d\times d$ matrices, $B_{r}(\text{Id})$ is the ball in $S_0^{d\times d}$ with center $\text{Id}$ and radius $r$, and $e_k=\frac{k}{|k|}$ for any $k\in \mathbb{R}^d$.
\begin{lemm}\label{geometric lemma}
     There exist $r_0>0$, $c_0>0$, a finite set $\Lambda\subset \mathbb{Z}^d$, smooth functions $\{\Gamma_k\in C^{\infty}(B_{r_o}(Id);\mathbb{R})\}_{k\in\Lambda}$ and smooth functions $\{\gamma_k\in C^{\infty}(\mathbb{R}^d;\mathbb{R})\}_{k\in\Lambda}$ such that
    \begin{gather*}
        \Gamma_k(R)\geq c_0,\quad\forall R\in B_{r_o}(\text{Id}) , \\
        R=\sum_{k\in\Lambda}\Gamma_k^2(R)e_k\otimes e_k,\quad\forall R\in B_{r_o}(\text{Id}),\\
        f=\sum_{k\in\Lambda}\gamma_k(f)e_k,\quad \forall f\in \mathbb{R}^d.
    \end{gather*}
\end{lemm}

\section{Antidivergence operator}
We recall two types of antidivergence operator $\mathcal{R}$ introduced in \cite{D1} and $\mathcal{G}$ introduced in \cite{T3}. 
\begin{prop}\label{def of antidiv}
    There exist two linear operators $\mathcal{R}:C^\infty(\mathbb{T}^d;\mathbb{R}^d)\rightarrow C^\infty(\mathbb{T}^d;S^{d\times d}_0)$ and $\mathcal{G}:C^\infty(\mathbb{T}^d;\mathbb{R})\rightarrow C^\infty(\mathbb{T}^d;\mathbb{R}^d)$ satisfying
    \begin{gather*}
        \text{div}\,\mathcal{R}v=v-\int_{\mathbb{T}^d}v,\quad\forall v\in C^\infty(\mathbb{T}^d;\mathbb{R}^d),\\
        \text{div}\,\mathcal{G}f=f-\int_{\mathbb{T}^d}f,\quad\forall f\in C^\infty(\mathbb{T}^d;\mathbb{R}).
    \end{gather*}
\end{prop}

Actually, $\mathcal{R}$ and $\mathcal{G}$ act as the degree -1 operator and we have the following theorem (see detailed proof in \cite[Theorem B.3]{serrin准则luo}.
\begin{prop}\label{theo of antidiv}
    Let $1\leq p\leq\infty$. For any vector field $v\in C^\infty(\mathbb{T}^d;\mathbb{R}^d)$ and any function $f\in C^\infty(\mathbb{T}^d;\mathbb{R})$, it holds that
    \begin{gather*}
        \|\mathcal{R}v\|_{L^p(\mathbb{T}^d)}\lesssim\|v\|_{L^p(\mathbb{T}^d)},
        \\\|\mathcal{G}f\|_{L^p(\mathbb{T}^d)}\lesssim\|f\|_{L^p(\mathbb{T}^d)}.
    \end{gather*}
    Moreover, if $v\in C^\infty_0(\mathbb{T}^d;\mathbb{R}^d)$ and $f\in C^\infty_0(\mathbb{T}^d;\mathbb{R})$, then there holds
    \begin{gather*}
        \|\mathcal{R}v(\sigma\cdot)\|_{L^p(\mathbb{T}^d)}\lesssim\sigma^{-1}\|v\|_{L^p(\mathbb{T}^d)},\quad\text{for any}\,\,\sigma\in\mathbb{N},
        \\\|\mathcal{G}f(\sigma\cdot)\|_{L^p(\mathbb{T}^d)}\lesssim\sigma^{-1}\|f\|_{L^p(\mathbb{T}^d)},\quad\text{for any}\,\,\sigma\in\mathbb{N}.
    \end{gather*}
\end{prop}

We introduce two bilinear operators $\mathcal{B}:C^{\infty}(\mathbb{T}^d,\mathbb{R}^d)\times C^{\infty}(\mathbb{T}^d,\mathbb{R}^{d\times d})\rightarrow C^{\infty}(\mathbb{T}^d,S_0^{d\times d})$ and $\tilde{\mathcal{B}}:C^{\infty}(\mathbb{T}^d,\mathbb{R})\times C^{\infty}(\mathbb{T}^d,\mathbb{R})\rightarrow C^{\infty}(\mathbb{T}^d,\mathbb{R}^d)$. The two operators share similar properties and permit us to obtain a derivative of the last argument.
we define two operators as follows.
\begin{gather*}
    \mathcal{B}(v,A)\triangleq v\mathcal{R}A-\mathcal{R}(\nabla v\mathcal{R}A),\quad\forall (v,A)\in C^{\infty}(\mathbb{T}^d,\mathbb{R}^d)\times C^{\infty}(\mathbb{T}^d,\mathbb{R}^{d\times d}),\\
    \tilde{\mathcal{B}}(f,g)\triangleq f\mathcal{G}g-\mathcal{G}(\nabla f\mathcal{G}g),\quad\forall (f,g)\in C^{\infty}(\mathbb{T}^d,\mathbb{R})\times C^{\infty}(\mathbb{T}^d,\mathbb{R}).
\end{gather*}

\begin{prop}\label{bilinear B}
    Let $1\leq p\leq \infty$. For any $(v,A)\in C^{\infty}(\mathbb{T}^d,\mathbb{R}^d)\times C^{\infty}(\mathbb{T}^d,\mathbb{R}^{d\times d})$ and $(f,g)\in C^{\infty}(\mathbb{T}^d,\mathbb{R})\times C^{\infty}(\mathbb{T}^d,\mathbb{R})$. The following holds,
    \begin{gather*}
        \text{div} \mathcal{B}(v,A)=vA-\fint_{\mathbb{T}^d}vA,\\
        \text{div} \tilde{\mathcal{B}}(f,g)=fg-\fint_{\mathbb{T}^d}fg,
    \end{gather*}
    and
    \begin{gather*}
        \|\mathcal{B}(v,A)\|_{L^p(\mathbb{T}^d)}\lesssim\|v\|_{C^1(\mathbb{T}^d)}\|\mathcal{R}A\|_{L^p(\mathbb{T}^d)},\\
         \|\tilde{\mathcal{B}}(f,g)\|_{L^p(\mathbb{T}^d)}\lesssim\|f\|_{C^1(\mathbb{T}^d)}\|\mathcal{G}g\|_{L^p(\mathbb{T}^d)}.
    \end{gather*}
\end{prop}
\begin{proof}
    This follows immediately from the definition of $\mathcal{B}$ and $\tilde{\mathcal{B}}$ and Proposition \ref{theo of antidiv}. Here we omit the proof.
\end{proof}

\section{Improved H\"{o}lder's inequality}
We recall the improved H\"{o}lder's inequality from \cite[Lemma 2.1]{modena} (see also \cite[Lemma3.7]{ns有限能量不唯一}).
\begin{lemm}\label{improved holder}
    Let $p\in[1,\infty]$ and $f,g:\mathbb{T}^d\rightarrow\mathbb{R}$ be smooth functions. Then for any $\sigma\in\mathbb{N}$, we have
    \begin{gather*}
        \|g\cdot f(\sigma\cdot)\|_{L^p}\lesssim\|g\|_{L^p}\|f\|_{L^p}+\sigma^{-\frac{1}{p}}\|g\|_{C^1}\|f\|_{L^p}.
    \end{gather*}
\end{lemm}

\smallskip
\noindent\textbf{Acknowledgments} This work was partially supported by the National Natural Science Foundation of China (No.12171493). 



\phantomsection
\addcontentsline{toc}{section}{\refname}
\bibliographystyle{abbrv} 
\bibliography{Reference}

\end{document}